\documentclass[reqno]{amsart}

\usepackage[all]{xy}
\usepackage{graphicx}

\usepackage{amssymb}
\usepackage{amsmath}
\usepackage{mathrsfs}
\usepackage{epsfig}
\usepackage{amscd}
\usepackage{graphicx, color}

\def\E{\ifmmode{\mathbb E}\else{$\mathbb E$}\fi} 
\def\N{\ifmmode{\mathbb N}\else{$\mathbb N$}\fi} 
\def\R{\ifmmode{\mathbb R}\else{$\mathbb R$}\fi} 
\def\Q{\ifmmode{\mathbb Q}\else{$\mathbb Q$}\fi} 
\def\C{\ifmmode{\mathbb C}\else{$\mathbb C$}\fi} 
\def\H{\ifmmode{\mathbb H}\else{$\mathbb H$}\fi} 
\def\Z{\ifmmode{\mathbb Z}\else{$\mathbb Z$}\fi} 
\def\P{\ifmmode{\mathbb P}\else{$\mathbb P$}\fi} 
\def\T{\ifmmode{\mathbb T}\else{$\mathbb T$}\fi} 
\def\SS{\ifmmode{\mathbb S}\else{$\mathbb S$}\fi} 
\def\DD{\ifmmode{\mathbb D}\else{$\mathbb D$}\fi} 

\newcommand{\del}{\partial}
\newcommand{\Cont}{{\operatorname{Cont}}}

\newcommand{\ben}{\begin{enumerate}}
\newcommand{\een}{\end{enumerate}}
\newcommand{\be}{\begin{equation}}
\newcommand{\ee}{\end{equation}}
\newcommand{\bea}{\begin{eqnarray}}
\newcommand{\eea}{\end{eqnarray}}
\newcommand{\beastar}{\begin{eqnarray*}}
\newcommand{\eeastar}{\end{eqnarray*}}
\newcommand{\bc}{\begin{center}}
\newcommand{\ec}{\end{center}}

\theoremstyle{theorem}
\newtheorem{thm}{Theorem}[section]
\newtheorem{cor}[thm]{Corollary}
\newtheorem{lem}[thm]{Lemma}
\newtheorem{prop}[thm]{Proposition}

\theoremstyle{definition}
\newtheorem{defn}[thm]{Definition}
\newtheorem{rem}[thm]{Remark}

\newtheorem{hypo}[thm]{Hypothesis}

\newtheorem*{thm*}{Theorem}

\numberwithin{equation}{section}

\def\R{{\mathbb R}}

\def\Crit{{\hbox{Crit}}}

\def\E{{\mathbb E}}
\def\Z{{\mathbb Z}}
\def\C{{\mathbb C}}
\def\R{{\mathbb R}}
\def\P{{\mathbb P}}

\def\N{{\mathbb N}}

\def\11{{\mathbb I}}

\def\delbar{{\overline \partial}}

\def\C{\mathbb{C}}
\def\Z{\mathbb{Z}}

\def\T{\mathbb{T}}

\def\Q{\mathbb{Q}}

\def\E{\ifmmode{\mathbb E}\else{$\mathbb E$}\fi} 
\def\N{\ifmmode{\mathbb N}\else{$\mathbb N$}\fi} 
\def\R{\ifmmode{\mathbb R}\else{$\mathbb R$}\fi} 
\def\Q{\ifmmode{\mathbb Q}\else{$\mathbb Q$}\fi} 
\def\C{\ifmmode{\mathbb C}\else{$\mathbb C$}\fi} 
\def\H{\ifmmode{\mathbb H}\else{$\mathbb H$}\fi} 
\def\Z{\ifmmode{\mathbb Z}\else{$\mathbb Z$}\fi} 
\def\P{\ifmmode{\mathbb P}\else{$\mathbb P$}\fi} 
\def\SS{\ifmmode{\mathbb S}\else{$\mathbb S$}\fi} 
\def\DD{\ifmmode{\mathbb D}\else{$\mathbb D$}\fi} 

\def\R{{\mathbb R}}

\def\Crit{{\hbox{Crit}}}
\def\E{{\mathbb E}}
\def\Z{{\mathbb Z}}
\def\C{{\mathbb C}}
\def\R{{\mathbb R}}

\def\N{{\mathbb N}}

\def\delbar{{\overline \partial}}

\def\CH{{\mathcal H}}

\def\CL{{\mathcal L}}
\def\CM{{\mathcal M}}

\def\CW{{\mathcal W}}

\def\darr#1{\raise1.5ex\hbox{$\leftrightarrow$}
\mkern-16.5mu #1}

\def\roughly#1{\raise.3ex\hbox{$#1$\kern-.75em
\lower1ex\hbox{$\sim$}}}

\def\opname#1{\mathop{\kern0pt{\rm #1}}\nolimits}

\def\dim{\opname{dim}}

\def\Cont{\operatorname{Cont}}
\def\Crit{\operatorname{Crit}}
\def\Spec{\operatorname{Spec}}
\def\Sing{\operatorname{Sing}}
\def\GFQI{\frak{G}}
\def\Index{\operatorname{Index}}

\def\Image{\operatorname{Image}}

\def\Int{\operatorname{Int}}

\begin{document}

\quad \vskip1.375truein

\def\mq{\mathfrak{q}}
\def\mp{\mathfrak{p}}
\def\mH{\mathfrak{H}}
\def\mh{\mathfrak{h}}
\def\ma{\mathfrak{a}}
\def\ms{\mathfrak{s}}
\def\mm{\mathfrak{m}}
\def\mn{\mathfrak{n}}
\def\mz{\mathfrak{z}}
\def\mw{\mathfrak{w}}
\def\Hoch{{\tt Hoch}}
\def\mt{\mathfrak{t}}
\def\ml{\mathfrak{l}}
\def\mT{\mathfrak{T}}
\def\mL{\mathfrak{L}}
\def\mg{\mathfrak{g}}
\def\md{\mathfrak{d}}
\def\mr{\mathfrak{r}}
\def\Cont{\operatorname{Cont}}
\def\Crit{\operatorname{Crit}}
\def\Spec{\operatorname{Spec}}
\def\Sing{\operatorname{Sing}}
\def\GFQI{\text{\rm GFQI}}
\def\Index{\operatorname{Index}}
\def\Cross{\operatorname{Cross}}
\def\Ham{\operatorname{Ham}}
\def\Fix{\operatorname{Fix}}
\def\Graph{\operatorname{Graph}}

\title[Contact instantons with Legendrian boundary condition]
{Contact instantons with Legendrian boundary condition: a priori estimates, asymptotic convergence
 and index formula}
\author{Yong-Geun Oh, Seungook Yu}
\address{Center for Geometry and Physics, Institute for Basic Science (IBS),
77 Cheongam-ro, Nam-gu, Pohang-si, Gyeongsangbuk-do, Korea 790-784
\& POSTECH, Gyeongsangbuk-do, Korea}
\email{yongoh1@postech.ac.kr}
\address{POSTECH, \&
Center for Geometry and Physics, Institute for Basic Science (IBS),
77 Cheongam-ro, Nam-gu, Pohang-si, Gyeongsangbuk-do, Korea 790-784}
\email{yso1460@postech.ac.kr}
\thanks{This work is supported by the IBS project \# IBS-R003-D1}


\begin{abstract}
In this paper, we establish nonlinear ellipticity
of the equation of contact instantons with Legendrian boundary condition
on punctured Riemann surfaces by proving the a priori elliptic coercive estimates for the
contact instantons with Legendrian boundary condition, and prove an asymptotic
exponential $C^\infty$-convergence result at a puncture under the uniform $C^1$ bound.
We prove that the asymptotic charge of contact instantons  at the punctures
\emph{under the Legendrian boundary condition} vanishes.
This eliminates the phenomenon of the appearance of
\emph{spiraling cusp instanton along a Reeb core}, which
removes the only remaining obstacle towards the compactification and the Fredholm
theory of the moduli space of contact instantons in the open string case, which plagues the
closed string case. Leaving the study of $C^1$-estimates and  details of Gromov-Floer-Hofer
style compactification of contact instantons to \cite{oh:entanglement1}, we also
derive an index formula which computes the virtual dimension of the moduli space.
These results are the analytic basis for the sequels \cite{oh:entanglement1}--\cite{oh:contacton-gluing}
and \cite{oh-yso:spectral} containing applications to contact topology
and contact Hamiltonian dynamics.
\end{abstract}


\maketitle

\tableofcontents

\section{Introduction}

This is the first part of a series of papers \cite{oh:entanglement1}--\cite{oh:contacton-gluing}
and \cite{oh-yso:spectral} and others in preparation
in which the (Hamiltonian-perturbed) contact instantons with Legendrian boundary condition and its
applications are studied. The motivation of our study of this problem is two-fold.
The first one is to construct a Floer theoretic
construction of Legendrian spectral invariants on the one-jet bundle given by
Th\'eret in \cite{theret}, Bhupal \cite{bhupal} and Sandon \cite{sandon:homology},
which are constructed using the generating functions quadratic at infinity ($\GFQI$).
This is the Legendrian version of Viterbo's $\GFQI$ spectral invariants
constructed in \cite{viterbo:generating} for the Lagrangian submanifolds in the cotangent bundle.

Recall that the Floer theoretic construction of Viterbo's invariant is given
by the present author in \cite{oh:jdg,oh:cag}. The starting point of
this Floer theoretic construction is a remarkable
observation of Weinstein \cite{alan:observation} which reads that \emph{the classical action functional
is a generating function of
the time-one image $\phi_H^1(0_{T^*B})$ of the zero section $0_{T^*B}$ under the Hamiltonian flow of $H = H(t,x)$.}

Therefore the natural first step towards Floer theoretic construction of Legendrian spectral invariants
is to find a similar formulation
of the contact version of Weinstein's observation. More precisely, let
$$
\lambda = dz - pdq
$$
be the standard contact one-form on $J^1B$ and $R = \psi_H^1(0_{J^1B})$ is the
Legendrian submanifold which is the time-one image of the contact flow $\psi_H^t$
associated to the time-dependent function $H = H(t,y)$ with $y = (x,z) \in J^1B$.
We have found the contact counterpart of Weinstein's observation whose detailed
explanation is given in \cite{oh-yso:weinstein}.

\subsection{Contact triads and contact instantons}

Let $(M, \xi)$ be a contact manifold.
A contact triad for the contact manifold $(M, \xi)$ is a triple $(M,\lambda, J)$
whose explanation is now in order. With $\lambda$ given, we have the Reeb vector field $R_\lambda$
uniquely determined by the equation $R_\lambda \rfloor d\lambda = 0, \, R_\lambda \rfloor \lambda = 1$.
Then we have decomposition $TM = \xi \oplus \R \{R_\lambda\}$. We denote by $\Pi: TM \to TM$
the associated idempotent whose image is $\xi$.
A \emph{CR almost complex structure $J$} is an endomorphism $J: TM \to TM$ satisfying $J^2 = - \Pi$
or more explicitly
$$
(J|_\xi)^2 = - id|_\xi, \quad J(R_\lambda) = 0.
$$
We say $J$ is adapted to $\lambda$ if $d\lambda(Y, J Y) \geq 0$ for all $Y \in \xi$ with
equality only when $Y = 0$. The associated Contact triad metric is given by
$$
g=g_\xi+\lambda\otimes\lambda.
$$
In \cite{oh-wang:connection}, the authors introduced the \emph{contact triad connection} associated to
every contact triad $(M, \lambda, J)$ with the contact triad metric and proved its existence and uniqueness.

Let $\dot \Sigma$ a boundary punctured Riemann surface associated a bordered compact Riemann surface
$(\Sigma, j)$. Then for a given map $u: \dot \Sigma \to M$, we can decompose its derivative
$du$, regarded as a $u^*TM$-valued one-form on $\dot \Sigma$, into
\be\label{eq:du}
du = d^\pi u + u^*\lambda \otimes R_\lambda
\ee
 where $d^\pi u := \Pi du$. Furthermore $d^\pi u$ is decomposed into
\be\label{eq:dpiu}
d^\pi u = \delbar^\pi u + \del^\pi u
\ee
where $\delbar^\pi u: = du^{\pi(0,1)}$ (resp. $\del^\pi u: = du^{\pi(1,0)}$) is
the anti-complex linear part (resp. the complex linear part) of $d^\pi u: (T\dot \Sigma, j) \to (\xi,J)$.

A contact instanton is
a map $w: \dot \Sigma \to M$ that satisfies the system of nonlinear partial differential equation
\be\label{eq:contacton-intro}
\delbar^\pi w = 0, \quad d(w^*\lambda \circ j) = 0
\ee
on a contact triad $(M,\lambda, J)$. In a series of papers,
\cite{oh-wang:CR-map1,oh-wang:CR-map2} jointed with Wang and in \cite{oh:contacton}, the first named author developed analysis of
contact instantons \emph{without taking symplectization}.

Then towards the construction of a Legendrian counterpart of the construction in \cite{oh:jdg,oh:cag}
we need to consider the boundary value problem of contact instanton equation \eqref{eq:contacton-intro}
with Legendrian boundary condition.

\subsection{A priori estimates for the Legendrian boundary value problem}
\label{subsec:apriori-estimates}

In the study of contact instantons in the \emph{closed string} context in \cite{oh-wang:CR-map1,oh-wang:CR-map2},
there is the phenomenon of the \emph{appearance of spiraling contact instantons along the Reeb core}, even for a
solution that has $C^1$-bound and finite $\pi$-energy
$$
E^\pi(u) = \frac12 \int_{\dot \Sigma} |d^\pi u|^2 < \infty.
$$
This is caused by a puncture at which the asymptotic charge $Q$ of a contact instanton is nonzero while the asymptotic
period $T$ is zero. (See \eqref{eq:TQ-T-intro} and \eqref{eq:TQ-Q-intro} below for the definitions of $T$ and $Q$.)
The main purpose of the present paper is to prove two fundamental analytical ingredients
for the construction of the moduli space of \emph{bordered contact instantons with the Legendrian boundary condition}
and prescribed asymptotic convergence at the boundary punctures of the domain Riemann surface $(\dot \Sigma, j)$.
One is the a priori elliptic regularity estimates and the other is the relevant Fredholm index formula
of the linearized operator in terms of certain topological index associated to the moduli space.

Let $(\Sigma,j)$ be a compact Riemann surface with boundary and $\dot \Sigma$ a
punctured Riemann surface with a finite number of boundary punctures. For the simplicity and
for the main purpose of the present paper, we focus on the genus zero case so that $\dot \Sigma$
is conformally the unit discs with boundary punctures $z_0, \ldots, z_k \in \del D^2$ ordered
counterclockwise, i.e.,
$$
\dot \Sigma \cong D^2 \setminus \{z_0, \ldots, z_k\}
$$
Then, for a  $(k+1)$-tuple $\vec R = (R_0,R_1, \cdots, R_k)$ of Legendrian submanifolds, which we call
an (ordered) Legendrian link, we consider
the boundary value problem
\be\label{eq:contacton-Legendrian-bdy-intro}
\begin{cases}
\delbar^\pi w = 0, \quad d(w^*\lambda \circ j) = 0,\\
w(\overline{z_iz_{i+1}}) \subset R_i
\end{cases}
\ee
as an elliptic boundary value problem for a map $w: \dot \Sigma \to M$
by deriving the a priori coercive elliptic estimates. Here $\overline{z_iz_{i+1}} \subset \del D^2$
is the open arc between $z_i$ and $z_{i+1}$.

\begin{rem}
We refer readers to \cite{oh:entanglement1,oh:perturbed-contacton-bdy} for another direction of
extending the equation to the following \emph{$X_H$-perturbed equation
with Legendrian boundary condition}
\be\label{eq:perturbed-contacton-bdy}
\begin{cases}
(du - X_H \otimes \gamma)^{\pi(0,1)} = 0, \quad d(e^{g_H(u)}(u^*\lambda + H\, \gamma)\circ j) = 0\\
u(z) \in \vec R \quad \text{\rm for } \, z \in \del \dot \Sigma.
\end{cases}
\ee
Here $H$ is a domain-dependent Hamiltonian, $\gamma$ is a one-form on $\dot \Sigma$
and the function $g_H(u): \dot \Sigma \to \R$ is a function that satisfies
\be\label{eq:gHu}
g_H(u)(t,x) := g_{((\psi_H^t (\psi_H^1)^{-1})^{-1}}(u(t,x))
\ee
near each puncture in terms of the strip-like coordinates.
It is established in \cite{oh:perturbed-contacton-bdy} that the equation \eqref{eq:perturbed-contacton-bdy}
 is also an elliptic boundary value problem, and utilized for the proof of
Sandon-Shelukhin's conjecture in contact topology. (See \cite{oh:shelukhin-conjecture}.)
\end{rem}

As the first step towards the analytic study of the
above boundary value problem \eqref{eq:contacton-Legendrian-bdy-intro},
we first show that the Legendrian boundary condition for the contact instanton is a
free boundary value problem, i.e., it satisfies
$$
\frac{\del w}{\del \nu} \perp TR
$$
for any Legendrian submanifold. (See \cite{jost:free-bdy} for
the importance of the free boundary value condition for a general study of
elliptic estimates of the minimal surface type equations.) Then we
prove the following elliptic $W^{1,2}$-estimate as an application of
Stokes' formula combined with the Legendrian boundary condition. The global tensorial calculation deriving
the a priori estimate illustrates how well the Legendrian boundary condition interacts with
triad connection and the contact instanton equation.

Using the decomposition \eqref{eq:du} and the properties $\xi \perp R_\lambda$ and $|R_\lambda| =1$
for the contact triad metric, we  can decompose the (full) harmonic energy density into
\be\label{eq:|dw|2}
|dw|^2 = |d^\pi w|^2 + |w^*\lambda|^2
\ee
\begin{thm}[$W^{2,2}$-estimates with boundary]\label{eq:W12-with-bdy}
Denote by $B$ the second fundamental form of the contact triad connection $\nabla$.
Let $w: \dot \Sigma\to M$ satisfy \eqref{eq:contacton-Legendrian-bdy-intro}.
Then for any relatively compact domains $D_1$ and $D_2$ in
$\dot\Sigma$ such that $\overline{D_1}\subset D_2$, we have
\be\label{eq:differential-inequality}
\|dw\|^2_{W^{1,2}(D_1)} \leq  C_1 \|dw\|^2_{L^2(D_2)} + C_2 \|dw\|^4_{L^4(D_2)}
+ \int_{\del D}|C(\del D)|
\ee
where we have the expression
\be\label{eq;CdelD}
C(\del D): = - 8 \left\langle B\left(\frac{\del w}{\del x},\frac{\del w}{\del x}\right),\frac{\del w}{\del y}\right \rangle
\ee
for any isothermal coordinates $z = x + iy$ adapted to $\del \dot \Sigma \cap D_2$.
\end{thm}
Out of this explicit form of the boundary contribution, we can easily derive the following
form of $W^{2,2}$-estimates, with some adjustment of constants $C_1, \, C_2$
(also with some cubic terms $\|dw\|^3_{L^2(D_2)}$ along the way), by some algebraic process of comparing the triad connection
$\nabla$ and the Levi-Civita connection $\nabla^{\text{\rm LC}}$ of a Riemannian metric
for which the Legendrian submanifold $R$ becomes totally geodesic. (See \cite[Section 8.2 \& 8.3]{oh:book1}
for the same arguments used for the case of pseudoholomorphic curves with Lagrangian boundary condition.)

\begin{thm}\label{thm:local-W12-intro}
Let $w: \dot \Sigma \to M$ satisfy \eqref{eq:contacton-Legendrian-bdy-intro}.
Then for any relatively compact domains $D_1$ and $D_2$ in
$\dot\Sigma$ such that $\overline{D_1}\subset D_2$, we have
$$
\|dw\|^2_{W^{1,2}(D_1)}\leq  C_4 \|dw\|^4_{L^4(D_2)}
$$
where $C_4$ is a constant depending only on $D_1$, $D_2$ and $(M,\lambda, J)$ and $R_i$'s.
\end{thm}

Once this $W^{2,2}$ estimates is established, we proceed with the higher regularity
estimates by first proving the following elliptic $W^{2,2}$-estimates in the general context of
maps $w: \dot \Sigma \to M$ for a Riemann surface with a finite number of boundary punctures.

Starting from Theorem \ref{thm:local-W12-intro} and using
the embedding $W^{2,2} \hookrightarrow C^{0,\delta}$ with $0 < \delta < 1/2$,
we also establish the following higher local $C^{k,\delta}$-estimates on punctured surfaces
$\dot \Sigma$ in terms of the $W^{2,2}$-norms with $\ell \leq k+1$.

\begin{thm}\label{thm:local-regularity-intro} Let $w$ satisfy \eqref{eq:contacton-Legendrian-bdy-intro}.
Then for any pair of domains $D_1 \subset D_2 \subset \dot \Sigma$ such that $\overline{D_1}\subset D_2$,
$$
\|dw\|_{C^{k,\delta}} \leq C_\delta(\|dw\|_{W^{1,2}(D_2)})
$$
where $C_\delta = C_\delta(r) > 0$ is a function continuous at $r = 0$
and depends only on $J$, $\lambda$ and $D_1, \, D_2$ but independent of $w$.
\end{thm}

With some adjustment of the function $C_\delta$, combining the two theorems, we obtain
\begin{cor} Let $w$ satisfy \eqref{eq:contacton-Legendrian-bdy-intro}.
Then for any pair of domains $D_1 \subset D_2 \subset \dot \Sigma$ such that $\overline{D_1}\subset D_2$,
$$
\|dw\|_{C^{k,\delta}} \leq C_\delta(\|dw\|_{L^4(D_2)})
$$
where $C_\delta = C_\delta(r) > 0$ is a function continuous at $r = 0$
and depends only on $J$, $\lambda$ and $D_1, \, D_2$ but independent of $w$.
\end{cor}
In particular, we prove that
any weak solution of \eqref{eq:contacton-Legendrian-bdy-intro} in
$W^{1,4}_{\text{\rm loc}}$ automatically becomes a classical solution.
(Compare \cite[Theorem 8.5.5]{oh:book1} for a similar theorem for the
Lagrangian boundary condition in symplectic geometry.)

\subsection{Asymptotic convergence and vanishing of asymptotic charge}

Next we study the asymptotic convergence result of contact instantons
of finite $\pi$-energy with Legendrian pair $(R_0,R_1)$ near the punctures of
a Riemann surface $\dot \Sigma$.

Let $\dot\Sigma$ be a punctured Riemann surface with punctures
$$
\{p^+_i\}_{i=1, \cdots, l^+}\cup \{p^-_j\}_{j=1, \cdots, l^-}
$$
equipped with a metric $h$ with \emph{strip-like ends} outside a compact subset $K_\Sigma$.
Let $w: \dot \Sigma \to M$ be any smooth map. As in \cite{oh-wang:CR-map1}, we define the total $\pi$-harmonic energy $E^\pi(w)$
by
\be\label{eq:endenergy}
E^\pi(w) = E^\pi_{(\lambda,J;\dot\Sigma,h)}(w) = \frac{1}{2} \int_{\dot \Sigma} |d^\pi w|^2
\ee
where the norm is taken in terms of the given metric $h$ on $\dot \Sigma$ and the triad metric on $M$.

We put the following hypotheses in our asymptotic study of the finite
energy contact instanton maps $w$ as in \cite{oh-wang:CR-map1}:
\begin{hypo}\label{hypo:basic-intro}
Let $h$ be the metric on $\dot \Sigma$ given above.
Assume $w:\dot\Sigma\to M$ satisfies the contact instanton equations \eqref{eq:contacton-Legendrian-bdy-intro},
and
\begin{enumerate}
\item $E^\pi_{(\lambda,J;\dot\Sigma,h)}(w)<\infty$ (finite $\pi$-energy);
\item $\|d w\|_{C^0(\dot\Sigma)} <\infty$.
\item $\Image w \subset \mathsf K\subset M$ for some compact subset $\mathsf K$.
\end{enumerate}
\end{hypo}

\begin{rem}
It is shown in \cite{oh:contacton}, \cite{oh:entanglement1} that the conditions
spelled out in this hypothesis can be achieved whenever one can avoid bubbling-off. However,
under this hypothesis and under the condition of exponential convergence to prescribed
asymptotic condition assumed later in the discussion of the moduli space of contact
instantons and its Fredholm theory, discussions of global energy and bubbling are
totally irrelevant for the purpose of the present paper, since calculation of the index is
done for a given smooth contact instanton $w$ with prescribed asymptotics for nondegenerate
pair $(\lambda, \vec R)$. (See \eqref{hypo:nondegeneracy} for the definition of the relevant
nondegeneracy condition.)
\end{rem}

Let $w$ satisfy Hypothesis \ref{hypo:basic-intro}. Under the nondegeneracy hypothesis
and asymptotic convergence at the punctures, we can associate two
natural asymptotic invariants at each puncture defined as
\bea
T & := & \lim_{r \to \infty} \int_{\{r\}\times [0,1]}(w|_{\{r\}\times [0,1] })^*\lambda
\label{eq:TQ-T-intro}\\
Q & : = & \lim_{r \to \infty} \int_{\{r\}\times [0,1]}((w|_{\{r\}\times [0,1] })^*\lambda\circ j)\label{eq:TQ-Q-intro}
\eea
at each puncture.
(Here we only look at positive punctures. The case of negative punctures is similar.)
As in \cite{oh-wang:CR-map1}, we call $T$ the \emph{asymptotic contact action}
and $Q$ the \emph{asymptotic contact charge} of the contact instanton $w$ at the given puncture.

\begin{thm}[Subsequence Convergence]\label{thm:subsequence-intro}
Let $w:[0, \infty)\times [0,1]\to M$ satisfy the contact instanton equations \eqref{eq:contacton-Legendrian-bdy-intro}
and Hypothesis \ref{hypo:basic-intro}.
Then for any sequence $s_k\to \infty$, there exists a subsequence, still denoted by $s_k$, and a
massless instanton $w_\infty(\tau,t)$ (i.e., $E^\pi(w_\infty) = 0$)
on the cylinder $\R \times [0,1]$  such that
$$
\lim_{k\to \infty}w(s_k + \tau, t) = w_\infty(\tau,t)
$$
in the $C^l(K \times [0,1], M)$ sense for any $l$, where $K\subset [0,\infty)$ is an arbitrary compact set.
Furthermore, $w_\infty$ has $Q = 0$ and the formula $w_\infty(\tau,t)= \gamma(T\, t)$  with
asymptotic action $T$, where $\gamma$ is some Reeb chord joining $R_0$ and $R_1$ of period $|T|$.
\end{thm}
\begin{cor}[Vanishing Charge] Assume the pair $(\lambda, \vec R)$ is nondegenerate.
Let $w$ be as above. Suppose that $w(\tau,\cdot)$ converges as $\tau \to \infty$ in the
strip-like coordinate at a puncture $p \in \del \Sigma$  with associated Legendrian pair $(R,R')$.
 Then its  asymptotic charge $Q$ vanishes at $p$.
 \end{cor}

\begin{rem}\label{rem:big-difference-intro}
The vanishing $Q = 0$ of the asymptotic charge of massless
contact instanton $w$ with (nonempty) Legendrian boundaries is a big deviation from
the closed string case studied in \cite{oh-wang:CR-map1,oh-wang:CR-map2,oh:contacton}.
This vanishing removes the only remaining obstruction, the so called the \emph{appearance of spiraling instantons along the
Reeb core}, to a Fredholm theory and to a compactification of the moduli space of contact instantons.
\end{rem}

The coercive elliptic regularity estimates and the vanishing theorem of asymptotic charge in the present
work will be the foundation of the analytic package of Gromov-Floer-Hofer compactification and
the Fredholm theory of Hamiltonian-perturbed contact instantons with Legendrian boundary condition and its applications in the sequels of the present paper \cite{oh:contacton-transversality}--\cite{oh:contacton-gluing} and \cite{oh-yso:spectral}.

\subsection{Fredholm theory and the index formula}
\label{subsec:index-formula}

Next, we study another crucial component, the relevant Fredholm theory and the index formula
for the equation \eqref{eq:contacton-Legendrian-bdy-intro}
by adapting the one from \cite{oh:contacton}, \cite{oh-savelyev}
to the current case of contact instantons \emph{with boundary}. The relevant Fredholm theory in general
has been developed  by the first named author in \cite{oh:contacton} for the closed string case and
\cite{oh:contacton-transversality} for the case with boundary.

In the present paper, we prove the index formula for general disk instanton $w$
with finite number of boundary punctures.

We recall  the following Fredholm property of the linearized operator that is proved in
\cite{oh:contacton-transversality}.

\begin{prop}[Proposition 3.18  \& 3.20 \cite{oh:contacton-transversality}] \label{prop:open-fredholm}Suppose that $w$ is a solution to \eqref{eq:contacton-Legendrian-bdy-intro}.
Consider the completion of $D\Upsilon(w)$,
which we still denote by $D\Upsilon(w)$, as a bounded linear map
from $\Omega^0_{k,p}(w^*TM,(\del w)^*T\vec R)$ to
$\Omega^{(0,1)}(w^*\xi)\oplus \Omega^2(\Sigma)$
for $k \geq 2$ and $p \geq 2$. Then
\begin{enumerate}
\item The off-diagonal terms of $D\Upsilon(w)$ are relatively compact operators
against the diagonal operator.
\item
The operator $D\Upsilon(w)$ is homotopic to the operator
\be\label{eq:diagonal}
\left(\begin{matrix}\delbar^{\nabla^\pi} + T_{dw}^{\pi,(0,1)}+ B^{(0,1)} & 0 \\
0 & -\Delta(\lambda(\cdot)) \,dA
\end{matrix}
\right)
\ee
via the homotopy
\be\label{eq:s-homotopy}
s \in [0,1] \mapsto \left(\begin{matrix}\delbar^{\nabla^\pi} + T_{dw}^{\pi,(0,1)} + B^{(0,1)}
& \frac{s}{2} \lambda(\cdot) (\CL_{R_\lambda}J)J (\pi dw)^{(1,0)} \\
s\, d\left((\cdot) \rfloor d\lambda) \circ j\right) & -\Delta(\lambda(\cdot)) \,dA
\end{matrix}
\right) =: L_s
\ee
which is a continuous family of Fredholm operators.
\item And the principal symbol
$$
\sigma(z,\eta): w^*TM|_z \to w^*\xi|_z \oplus \Lambda^2(T_z\Sigma), \quad 0 \neq \eta \in T^*_z\Sigma
$$
of \eqref{eq:diagonal} is given by the matrix
\beastar
\left(\begin{matrix} \frac{\eta + i\eta \circ j}{2} Id  & 0 \\
0 & |\eta|^2
\end{matrix}\right).
\eeastar
\end{enumerate}
\end{prop}
Then we have the Fredholm index of $D\Upsilon(w)$ is given by
$$
\operatorname{Index}D\Upsilon(w) =
\operatorname{Index}(\delbar^{\nabla^\pi} + T_{dw}^{\pi,(0,1)}+ B^{(0,1)}) +
\operatorname{Index}(-\Delta_0).
$$

For this purpose of computing the Fredholm index of the linearized operator in terms of a topological index, especially in terms of those who will equip the relevant Floer-type
complex with the absolute grading, we need  the notion of anchored Legendrian
submanifolds which is an adaptation of that of Lagrangian submanifolds studied in \cite{fooo:anchored}
in symplectic geometry.
\begin{defn}\label{defn:anchor-intro}
Fix a base point $y$ of ambient contact manifold $(M,\xi)$. Let $R$ be a Legendrian submanifold of $(M,\xi)$. We define an
\emph{anchor} of $R$ to $y$ is a path $\ell:[0,1] \rightarrow M$ such that $\ell(0)=y,\;\ell(1)\in R$. We call a pair $(R,\ell)$
an \emph{anchored} Legendrian submanifold.
A chain $\mathcal{E} = ((R_0,\ell_0),\ldots,(R_k,\ell_k))$ is called an \emph{anchored Legendrian chain}.
\end{defn}

In order to express the analytical index of the linearized problem of \eqref{eq:contacton-Legendrian-bdy-intro}
in terms of a topological index of the Maslov-type, we need to equip an additional data with the anchor as in \cite{fooo:anchored}.
For the above given base point $y\in M$, we fix a Lagrangian subspace $V_y\in Lag(\xi_y)$.

\begin{defn}
Consider an anchored Legendrian $(R,\ell)$. We denote by $\alpha$ a section of $\ell^*Lag(\xi,d\lambda)$ such that
$$
\alpha(0)=V_y, \quad \alpha(1)=T_{\ell(1)}R \subset \xi_{\ell(1)}
$$
We call such a pair $(\ell,\alpha)$ a \emph{graded anchor} of $R$ (relative to $(y,V_y)$) and a triple $(R,\ell,\alpha)$
a \emph{graded anchored Legendrian submanifold}.
\end{defn}

Let $\vec R = (R_0, \ldots, R_k)$ be a Legendrian link be given.
Let $\gamma^+_i$ for $i =1, \cdots, s^+$ and $\gamma^-_j$ for $j = 1, \cdots, s^-$
be two given collections of Reeb chords at positive and negative punctures
respectively. Following the notations from \cite{behwz},
we denote by $\underline \gamma$ and $\overline \gamma$ the corresponding
collections
\beastar
\underline \gamma & = & \{\gamma_1^+,\cdots, \gamma^+_{s^+}\} \\
\overline \gamma & = & \{\gamma_1^-,\cdots, \gamma^-_{s^-}\}.
\eeastar
For each boundary puncture, we associate a strip-like
coordinates $(\tau,t) \in [0,\infty) \times [0,1]$ (resp. $(\tau,t) \in (-\infty,0] \times S^1$)
on the neighborhoods of boundary punctures.

We consider the set of smooth maps satisfying the boundary condition
\be\label{eq:bdy-condition}
w(z) \in R_i \quad \text{ for } \, z \in \overline{z_{i-1}z_i} \subset \del \dot \Sigma
\ee
and the asymptotic condition
\be\label{eq:asymp-condition}
\lim_{\tau \to \infty}w((\tau,t)_i) = \gamma^+_i(T_i(t+t_i)), \qquad
\lim_{\tau \to - \infty}w((\tau,t)_j) = \gamma^-_j(T_j(t-t_j))
\ee
for some $t_i, \, t_j \in S^1$, where
$$
T_i = \int_{S^1} (\gamma^+_i)^*\lambda, \, T_j = \int_{S^1} ( \gamma^-_j)^*\lambda.
$$
Here $t_i,\, t_j$ depends on the given analytic coordinate and the parameterizations of
the Reeb chords.

Fix a domain complex structure $j$ and its associated K\"ahler metric $h$.
We regard the assignment
$$
\Upsilon: w \mapsto \left(\delbar^\pi w, d(w^*\lambda \circ j)\right), \quad
\Upsilon: = (\Upsilon_1,\Upsilon_2)
$$
as a section of a (infinite dimensional) vector bundle and consider
the linearized operator
$$
D\Upsilon(w): \Omega^0(w^*TM,(\del w)^*T\vec R) \to \Omega^{(0,1)}(w^*\xi) \oplus \Omega^2(\Sigma).
$$
(See \cite{oh:contacton-transversality} for the precise formulation and Section \ref{sec:off-shell}
for a brief summary.)

The following index formula then is derived in Section \ref{sec:index}
by combining the (linear) gluing formula of the Fredholm index and
an explicit calculations of $\operatorname{Index}(\delbar^{\nabla^\pi} + T_{dw}^{\pi,(0,1)}+ B^{(0,1)})$
and $\operatorname{Index}(-\Delta_0$ on the semi-strips after a suitable
homotopies. The calculation of the former is
similar to the one as done in \cite{fooo:anchored}. However the latter calculation for the
Laplacian $-\Delta_0$ is a new element to be studied. (See Section \ref{sec:index} for details.)

\begin{thm}[Theorem \ref{thm:index-formula}]\label{thm:index-intro}
Let $w$ be a solution of \eqref{eq:contacton-Legendrian-bdy-intro}, i.e., $\Upsilon(w) = 0$.
Let $D\Upsilon(w)$ be the linearization operator of \eqref{eq:contacton-Legendrian-bdy-intro}.
Then we have the Fredholm index
$$
\operatorname{Index}D\Upsilon(w) = n - \sum_{i=1}^{k}\mu([w_{(i-1)i}^+,\gamma_i];\alpha_{(i-1)i}).
$$
\end{thm}
Here $\mu([w,\gamma_i];\alpha_{(i-1)i})$ is a  certain
Maslov-type index whose precise definition we refer to Section \ref{sec:polygonal}.
This index will be made more explicit in \cite{oh-yso:spectral}
for the case of Hamiltonian isotopes of the zero section of the one-jet bundle,
which will give rise to one that is the analog to the formula given in \cite{oh:jdg}
for the cotangent bundle case in symplectic geometry.

\subsection{Relationships with other literature on pseudoholomorphic curves on symplectization}
\label{subsec:comparison}

We would like to mention that a study of the relevant analysis of pseudoholomorphic
curves via the symplectization has been done by Hofer-Wysocki-Zehnder
in 3 dimensions in a series of papers.
The papers \cite{HWZ:asymptotics,HWZ:asymptotics-correction}
are especially relevant to the content of the present paper.

As far as the local  regularity result, Theorem 1.3,  in the exact case, i.e.,
$w^*\lambda \circ j = df$ is concerned,  it can be derived from the local regularity result from \cite{HWZ:asymptotics,HWZ:asymptotics-correction}
without much difficulty.  And the asymptotic convergence
result from \cite{HWZ:asymptotics,HWZ:asymptotics-correction} or \cite{wendl:lecture}
(see also \cite{bourgeois}), after the study of the relevant boundary value problem
for the symplectization of  the pair $(M,R)$, it can be also derived from the asymptotic study given in \cite{HWZ:asymptotics,HWZ:asymptotics-correction}  and \cite{wendl:lecture}.

In those cases, one may regard our proofs are just
different global tensorial proofs  in the spirit of geometric analysis
commonly exercised in Riemannian geometry.  However we would like to
point out that the regularity statement of Theorem 1.3 is of
different nature from and stronger than those presented in
the literature  in the study of pseudoholomorphic curves in symplectization:
\emph{our estimate on $w$ is completely independent of the `radial component',
which does not exist in our study, while the regularity study of
pseudoholomorphic curves in symplectization mixes up that of contact part
and that of radial part.} (This remark is already made in the introduction of
\cite{oh-wang:CR-map1}.) We anticipate that  by starting from
the study of contact instantons, this kind of estimate will be useful
for a finer study of moduli space of pseudoholomorphic curves in symplectization.
For example, we believe that we can construct a Kuranishi structure of
the moduli space of  contact instantons which admits a certain morphism to that of
the moduli space of pseudoholomorphic curves in the symplectization. We suspect
this would be useful in the full framework of symlectic field theory.
(See \cite{oh:contacton-gluing} for  some indication thereof.)

Our regularity result automatically gives
rise to that of radial component $f$ therein by integration of the one-form $w^*\lambda \circ j$
for the exact case $w^*\lambda \circ j = df$. To set the record straight and demonstrate
our claim, we hope to rewrite the analysis of pseudoholomorphic curves on
symplectization elsewhere, starting from that of contact instantons we have developed.
In fact, this has been already explicitly done in \cite{oh-savelyev} for the
Fredholm theory in \cite[Theorem 1.15]{oh-savelyev}, when the authors of the paper study
pseudoholomorphic curves on the $\frak{lcs}$-fication of contact manifold. The case of symplectization
is a special case therein where the charge class  defined in \cite{oh-savelyev}
vanishes, or we may say that the symplectization corresponds to the
`zero-temperature limit' of the $\frak{lcs}$-fication.

Our global tensorial approach has already demonstrated its usefulness and effectiveness
in the study of problems on contact Hamiltonian dynamics, especially through
some global calculus of variation and contact Hamiltonian calculus and dynamics.
(See \cite{oh:entanglement1}, \cite{oh:shelukhin-conjecture}.)
This global approach combined with well-known strategy of using
Bochner-Weitezenb\"ock type formulae is crucial in the analytic foundation of
perturbed contact instanton equation \eqref{eq:perturbed-contacton-bdy}, which is
the contact counterpart of Floer's Hamiltonian perturbed trajectory equation.
(See \cite{salamon-zehnder} for the corresponding analytical study of
the latter equation and its application.)

We would like to mention that
a study of relevant analysis of pseudoholomorphic curves  via symplectization
in the relative setting is given by Ekholm-Etynre-Sullivan \cite{EES-jdg,EES-tams} in the
context of relative Legendrian contact homology:
They restrict themselves to the case \emph{when the contact manifold $M$
admits a projection} $\pi: M \to P$ for some exact symplectic manifold $(P,d\alpha)$
and the contact form $\lambda$ is given by $\lambda = dt + \alpha$ on $M$.
In this regard, our study extends their results therefrom to the case of arbitrary pair
$(M,R)$ of contact manifold $M$ and  Legendrian submanifold $R$ which are either
compact or tame in the sense of \cite{oh:entanglement1}.

\medskip

Now comes the organization of the paper in order. The paper consists of 2 parts. Part \ref{part:estimates} consists of the results of a priori estimates, and of asymptotic convergence
and vanishing charge at the boundary punctures.
 In Section \ref{sec:local-W12} and \ref{sec:Ckdelta-estimates},
we establish the local $W^{2,2}$ and $C^{k,\delta}$ estimates for $k \geq 1$ with $0< \delta < 1/2$ respectively.  In Section \ref{sec:subsequence-convergence}
we prove that the presence of Legendrian barrier forces the asymptotic charge to vanish.
In Section \ref{sec:exponential-convergence}, we prove that when the Legendrian boundary
pair $(R_0,R_1)$ near a puncture are transversal in a suitable sense then the above convergence
 is exponentially fast.
In Part \ref{part:index}, we set up the Fredholm theory for the contact instantons with Legendrian boundary condition and prove the relevant index formula.

Part I of the present article has been circulated as a part of the first named author's arXiv
posting arXiv:2103.15390 entitled ``Contact Hamiltonian dynamics
and perturbed contact instantons  with Legendrian boundary
condition''. The present article  not only provides many more details of the proofs
thereof but also the index formula for the equation \eqref{eq:contacton-Legendrian-bdy-intro}
added, both of which are essential ingredients
for the general study of the moduli space of solutions and   and its applications to contact
topology and  contact Hamiltonian dynamics.
(See \cite{oh:entanglement1,oh-yso:spectral,oh:entanglement2} for example.)
We refer readers to
\cite{oh:entanglement1} for the study of compactification of the moduli space of
bordered contact instantons,  to \cite{oh:shelukhin-conjecture}
for its application to the study of Sandon-Shelukhin's conjecture on
translated points of contactomorphisms, to \cite{oh:contacton-transversality} for the gluing theory of
contact instantons and to \cite{oh:perturbed-contacton-bdy} for the extension of the study of
the present paper to that of Hamiltonian-perturbed (bordered) contact instantons.

Answering to popular requests, we add Subsection \ref{subsec:comparison}
which contains explanation on the relationship between
our analysis of contact instantons and that of pseudoholomorphic curves in symplectization,
and an appendix which contains a  summary of basic exterior
 differential calculus  of vector valued forms in Appendix \ref{append:weitzenbock}.

\begin{rem} \begin{enumerate}
\item
Along the way, we would like to make the written statement
for the first time since the appearance of contact triad connection introduced in
 \cite{oh-wang:connection} that the triad connection  suits the best for the effective tensorial study
of (perturbed) contact instantons.
(See Remark \ref{rem:triad-connection} for further detailed comments.)
\item
We refer readers to \cite{oh-savelyev} for analog for the study of
pseudoholomorphic curves on the $\frak{lcs}$-fication of contact manifolds which also applies to the case of symplectization.
\end{enumerate}
\end{rem}

\bigskip

\noindent{\bf Acknowledgement:} We would like to express our sincere thanks
to unknown referees for pointing out many careless mistakes and making useful suggestions
on the presentation of the paper, which we believe greatly improves
the exposition of the paper.

\bigskip

\noindent{\bf Convention:}

\medskip

\begin{itemize}
\item {(Contact Hamiltonian)} The contact Hamiltonian of a time-dependent contact vector field $X_t$ is
given by
$$
H: = - \lambda(X_t).
$$
We denote by $X_H$ the contact vector field whose associated contact Hamiltonian is given by $H = H(t,x)$.
\item The Hamiltonian vector field on symplectic manifold $(P, \omega)$ is defined by $X_H \rfloor \omega = dH$.
We denote by
$$
\psi_H: t\mapsto \psi_H^t
$$
its Hamiltonian flow.
\item We denote by $R_\lambda$ the Reeb vector field associated to $\lambda$. We denote by
$$
\phi_{R_\lambda}^t
$$
its flow. Then we have $\phi_{R_\lambda}^t  = \psi_{-1}^t$, i.e., \emph{the Reeb flow is the contact Hamiltonian
flow of the constant Hamiltonian $H = -1$}.
\item Throughout the paper,  various positive constants denoted by
 $C$ or  $C_i$'s and a function $C_\delta = C_\delta(r), \, C_\delta(r,s)$ and
 $C_\delta(r_1, r_2, \ldots, r_k)$  appear which vary place by place.
 They are independent of the smooth map $w$'s
  but depend only on the geometry of triads $(M,\lambda, J)$, Legendrian boundary condition
  and relevant subdomains of $w$ considered in the estimates. In the case of $C_\delta$
  appearing in Section \ref{sec:Ckdelta-estimates},  it is always continuous at
  $\vec r = (r_1, \cdots, r_k) = (0, \ldots, 0)$ and defined for all sufficiently small
$\vec r$ whose smallness depends on the bounds  obtained in the previous stage
of inductive procedure of alternating bootstraps performed in Section \ref{sec:Ckdelta-estimates}.
\end{itemize}

\part{A priori estimates, asymptotic convergence and charge vanishing}
\label{part:estimates}

Consider general contact manifolds $(M,\lambda)$ equipped with a $\lambda$-adapted
CR-almost complex structure called a \emph{contact triad} in \cite{oh-wang:connection,oh-wang:CR-map1,oh-wang:CR-map2}.

To highlight the main points of the a priori estimates of perturbed contact instantons \emph{with
Legendrian boundary condition}, we will restrict to the case with $H = -1$ in the rest of the present
paper postponing the discussion of the modifications needed to handle the Hamiltonian term to
 other sequels \cite{oh:perturbed-contacton-bdy}, \cite{oh-yso:spectral}.

We start with a summary of the basic properties of the contact triad connection introduced in \cite{oh-wang:connection}
which will be used in our tensor calculations in deriving various Weitzenb\"ock-type formulae. The defining properties of
the triad connection also interact very well with the Legendrian boundary condition for the contact instantons.
(See Section \ref{sec:local-W12} to see how.)

\section{Review of the contact triad connection}
\label{subsec:connection}

Assume $(M, \lambda, J)$ is a contact triad for the contact manifold $(M, \xi)$, and equip it with
 the contact triad metric
$g=g_\xi+\lambda\otimes\lambda$.
In \cite{oh-wang:connection}, the authors introduced the \emph{contact triad connection} associated to
every contact triad $(M, \lambda, J)$ with the contact triad metric and proved its existence and uniqueness.

\begin{thm}[Contact Triad Connection \cite{oh-wang:connection}]\label{thm:connection}
For every contact triad $(M,\lambda,J)$, there exists a unique affine connection $\nabla$, called the contact triad connection,
 satisfying the following properties:
\begin{enumerate}
\item The connection $\nabla$ is  metric with respect to the contact triad metric, i.e., $\nabla g=0$;
\item The torsion tensor $T$ of $\nabla$ satisfies $T(R_\lambda, \cdot)=0$;
\item The covariant derivatives satisfy $\nabla_{R_\lambda} R_\lambda = 0$, and $\nabla_Y R_\lambda\in \xi$ for any $Y\in \xi$;
\item The projection $\nabla^\pi := \pi \nabla|_\xi$ defines a Hermitian connection of the vector bundle
$\xi \to M$ with Hermitian structure $(d\lambda|_\xi, J)$;
\item The $\xi$-projection of the torsion $T$, denoted by $T^\pi: = \pi T$ satisfies the following property:
\be\label{eq:TJYYxi}
T^\pi(JY,Y) = 0
\ee
for all $Y$ tangent to $\xi$;
\item For $Y\in \xi$, we have the following
$$
\del^\nabla_Y R_\lambda:= \frac12(\nabla_Y R_\lambda- J\nabla_{JY} R_\lambda)=0.
$$
\end{enumerate}
\end{thm}
From this theorem, we see that the contact triad connection $\nabla$ canonically induces
a Hermitian connection $\nabla^\pi$ for the Hermitian vector bundle $(\xi, J, g_\xi)$, and we call it the \emph{contact Hermitian connection}. This connection will be used to study estimates for the $\pi$-energy in later sections.

Moreover, the following fundamental properties of the contact triad connection was
proved in \cite{oh-wang:connection}, which will be useful to perform tensorial calculations later.

\begin{cor}\label{cor:connection}
Let $\nabla$ be the contact triad connection. Then
\begin{enumerate}
\item For any vector field $Y$ on $M$,
\be\label{eq:nablaYX}
\nabla_Y R_\lambda = \frac{1}{2}(\CL_{R_\lambda}J)JY;
\ee
\item $\lambda(T|_\xi)=d\lambda$.
\end{enumerate}
\end{cor}

We refer readers to \cite{oh-wang:connection} for more discussion on the contact triad connection and its relation with other related canonical type connections.

\section{Local coercive $W^{2,2}$ estimate under Legendrian boundary condition}
\label{sec:local-W12}

Now we establish the elliptic a priori estimates for the boundary
value problem \eqref{eq:contacton-Legendrian-bdy-intro} which extends the interior estimates proved
in \cite{oh-wang:CR-map1}.

\subsection{Basic differential inequalities from \cite{oh-wang:CR-map1}}

We start with the following pointwise inequality derived in \cite{oh-wang:CR-map1}.
\begin{lem}[Equation (5.13), \cite{oh-wang:CR-map1}]\label{lem:second:derivative}
Let $w$ be any contact instanton $w: (\dot \Sigma,j) \to (M; \lambda, J)$ i.e., any
map satisfying the equation
\be\label{eq:contacton}
\delbar^\pi w = 0 , \quad d(w^*\lambda \circ j) = 0.
\ee
Then we have
\be\label{eq:second-derivative}
|\nabla(dw)|^2 \leq C_1 |dw|^4 - 4K |dw|^2 - 2\Delta |dw|^2
\ee
\end{lem}
We recall from \eqref{eq:|dw|2} that
$$
|dw|^2 = |d^\pi w|^2 + |w^*\lambda|^2 = |\del^\pi w|^2 + |w^*\lambda|^2
$$
for the contact instanton $w$ which satisfies $\delbar^\pi w = 0$.

To illustrate usefulness of \eqref{eq:contacton} and as a warm-up for the
tensorial calculations we will soon carry out, we recall the proof of
the following $\epsilon$-regularity and interior density estimate
which were proved in \cite{oh-wang:CR-map1}.

\begin{thm}[Corollary 5.2 \cite{oh-wang:CR-map1}]\label{thm:density}
There exist constants $C, \, \epsilon_0$ and $r_0 > 0$, depending only on $J$ and
the Hermitian metric $h$ on $\dot \Sigma$, such that for any
 $C^1$ contact instanton $w: \dot \Sigma \to M$ with
$$
E(r_0): = \frac{1}{2}\int_{D(r_0)} |dw|^2 \leq \epsilon_0,
$$
and discs $D(2r) \subset \operatorname{Int}\Sigma$ with $0 < 2r \leq r_0$,
$w$ satisfies
\be\label{eq:schoen's}
\max_{\sigma \in (0,r]} \left(\sigma^2 \sup_{D(r-\sigma)}
e(w)\right) \leq CE(r)
\ee
for all $0< r \leq r_0$. In particular, letting $\sigma = r/2$, we obtain
\be\label{eq:supeu}
\sup_{D(r/2)} |dw|^2 \leq \frac{4C E(r)}{r^2}
\ee
for all $r \leq r_0$.
\end{thm}

The proof of this theorem is an immediate consequence of \eqref{eq:contacton}
by rewriting \eqref{eq:contacton} into
$$
\Delta |dw|^2 \leq - \frac12 \nabla |dw|^2 + \frac{C_1}{2} |dw|^4 - 2K |dw|^2 - \Delta |dw|^2
\leq  \frac{C_1}{2} |dw|^4 - 2K |dw|^2
$$
which gives rise to the following differential inequality
$$
\Delta e(w)\leq Ce(w)^2+\|K\|_{L^\infty(\dot\Sigma)}e(w),
$$
with
$$
C=2\|\CL_{R_\lambda}J\|^2_{C^0(M)}+\|\nabla^\pi(\CL_{R_\lambda}J)\|_{C^0(M)}+\|\text{\rm Ric}\|_{C^0(M)}+1
$$
which is a positive constant independent of $w$. By considering it on a sufficiently small domain,
the theorem will then follow by the standard argument
from \cite{schoen} (See also the proof of \cite[Theorem 8.13]{oh:book1} and
 \cite[Theorem 5.1]{oh-wang:CR-map1}.)

Next we will again utilize \eqref{eq:contacton} for the local
boundary a priori estimates, in which the Legendrian boundary condition
and the usage of contact triad connection play important roles.

\subsection{$W^{1,2}$ boundary estimate}

The main goal of the present section is to derive the following local boundary a priori estimate.

\begin{thm}\label{thm:local-W12} Let $w: \R \times [0,1] \to M$ satisfy \eqref{eq:contacton-Legendrian-bdy-intro}.
Then for any relatively compact domains $D_1$ and $D_2$ in
$\dot\Sigma$ such that $\overline{D_1}\subset D_2$, we have
\be\label{eq:without-bdyterm}
\|dw\|^2_{W^{1,2}(D_1)}\leq C_1 \|dw\|^2_{L^2(D_2)} + C_2 \|dw\|^4_{L^4(D_2)}
\ee
where $C_1, \ C_2$ are some constants which
depend only on $D_1$, $D_2$ and $(M,\lambda, J)$ and $C_3$ is a
constant which also depends on $R_i$ with $w(\del D_2) \subset R_i$ as well.
\end{thm}

Before delving into the tensorial calculation, we outline the strategy of the proof
following the same strategy used for  pseudoholomorphic curves with Lagrangian boundary condition
presented in \cite[Section 8.2 \& 8.3]{oh:book1}:
\begin{enumerate}
\item As the first step, we utilize the contact triad connection $\nabla$ for the study of
boundary value problem to derive the following differential inequality
\be\label{eq:differential-inequality}
\|dw\|^2_{W^{1,2}(D_1)} \leq  C_1 \|dw\|^2_{L^2(D_2)} + C_2 \|dw\|^4_{L^4(D_2)}
+ \int_{\del D}|C(\del D)|
\ee
where we have
\be\label{eq;CdelD}
C(\del D): = - 8\left\langle B\left(\frac{\del w}{\del x},\frac{\del w}{\del x}\right),\frac{\del w}{\del y}\
\right \rangle
\ee
with isothermal coordinate $z = x+ iy$ adapted to $\del \dot \Sigma \cap D_2$.
\item Once we derive \eqref{eq:differential-inequality}, noting that Legendrian boundary
condition $R$ for the contact instanton is automatically the free boundary value problem, i.e.,
$$
\frac{\del w}{\del \nu} \perp TR,
$$
one can use the Levi-Civita connection of a metric for which $R$ becomes
\emph{totally geodesic} (i.e., $B=0$) which will eliminate the boundary contribution
appearing above in \eqref{eq:differential-inequality}. Then recalling the standard fact
that $\nabla = \nabla^{\text{\rm LC}} + P$ for a $(2,1)$ tensor, we can convert the inequality
into \eqref{eq:without-bdyterm}. (See \cite[Section 8.3]{oh:book1} for such detail.)
\end{enumerate}

Therefore we will focus on the derivation of the inequality \eqref{eq:differential-inequality}
in the rest of the section.

\begin{proof} [Proof of Theorem \ref{thm:local-W12}]
The first part of the proof is similar to
the interior estimate of \cite[Proposition 5.3]{oh-wang:CR-map1} given in
\cite[Appendix C]{oh-wang:CR-map1}.

However  to handle the Legendrian boundary condition in the estimates,
we need to carefully analyze how the contact instanton equation interacts with
Legendrian boundary. This turns out to be another place where the full power of
defining properties of contact triad connection is exercised.

We have only to consider the case of a pair of semi-discs $D_1,\, D_2 \subset \dot \Sigma$
with $\overline D_1 \subset D_2$ such that $\del D_2 \subset \del D_1 \subset \del \dot \Sigma$.
(The open disc cases of $D_1 \subset D_2$ are already treated in \cite[Appendix C]{oh-wang:CR-map1}.)

For the pair of given domains $D_1$ and $D_2$, we choose another domain $D$ such that
$\overline D_1 \subset D \subset \overline D \subset D_2$ and a smooth cut-off function $\chi:D_2\to \R$ such that
$\chi\geq 0$ and
$\chi\equiv 1$ on $\overline{D_1}$, $\chi\equiv 0$ on $D_2-D$. Then we have
$$
\int_{D_1}|\nabla(dw)|^2 \leq \int_{D}\chi^2|\nabla(dw)|^2.
$$
Multiplying  \eqref{eq:second-derivative} by $\chi^2$ and integrating over $D$, we
also get
\bea\label{eq:intchi2}
\int_{D}\chi^2|\nabla(dw)|^2
&\leq&C_1\int_{D}\chi^2|dw|^4-4\int_{D}K\chi^2|dw|^2-2\int_{D}\chi^2\Delta e \nonumber\\
&\leq&C_1\int_{D_2}|dw|^4+4\|K\|_{L^\infty(\dot\Sigma)}\int_{D_2}|dw|^2-2\int_{D}\chi^2\Delta e
\eea
where $C_1$ is the same constant as the one appearing in \eqref{eq:second-derivative}.

We now deal with the last term $\int_{D_2}\chi^2 \Delta e$.
We rewrite
\bea\label{eq:chi2Deltae}
\chi^2\Delta e\, dA&=&*(\chi^2 \Delta e)=\chi^2 *\Delta e
=-\chi^2 d*de \nonumber\\
&=&-d(\chi^2 *de)+2\chi d\chi\wedge (*de).
\eea
To deal with the right hand side, we estimate the two terms separately.

For the second term of \eqref{eq:chi2Deltae}, we derive the following inequality.
(See \cite[p.677]{oh-wang:CR-map1} for the details of its derivation.)

\begin{lem}[Equation (3.7), \cite{oh-wang:CR-map1}]
For any $\epsilon > 0$, we have
\be\label{eq:second-term}
\left|\int_{D}\chi d\chi\wedge(*de)\right| \leq
\frac{1}{\epsilon}\int_{D}\chi^2|\nabla(dw)|^2\,dA+\epsilon\|d\chi\|_{C^0(D)}^2\int_{D}|dw|^2\,dA
\ee
\end{lem}

We now estimate the integral of the first term of \eqref{eq:chi2Deltae}
\be\label{eq:first-term}
\int_D -d(\chi^2 * de) = \int_{\del D} - \chi^2 *de
\ee
by Stokes' formula. For this purpose, we need to calculate $*de|_{\del D}$ explicitly.
\emph{This is where the new element, the boundary estimate, starts the
calculation of which did not exist in \cite{oh-wang:CR-map1}. It is also a place where the full
power of the defining properties of the contact triad connection enters.}

Let $z_0 \in \dot \Sigma$ be any boundary point and $z = x+iy$ be
an isothermal coordinate of $(\dot \Sigma, j)$ so that $h = dx^2 + dy^2$
on a semi-disc open neighborhood $D_2 \subset \dot \Sigma$ with $\del D^2 \subset \del \dot \Sigma$
and $\frac{\del}{\del x} \in TR$ on $\del D^2$.

\begin{lem}\label{lem:Neunman-bdy} $w$ satisfies the Neunmann boundary condition,
i.e., $\frac{\del w}{\del y} \perp TR_i$.
\end{lem}
\begin{proof} Since $R_i$ are Legendrian, $TR \subset \xi$ and hence
$$
\frac{\del w}{\del x} = \left(\frac{\del w}{\del x}\right)^\pi
$$
along $\del D_2$.
Since $\delbar^\pi w= 0$, we have
$$
- J\left(\frac{\del w}{\del y}\right)^\pi
= \left(\frac{\del w}{\del x}\right)^\pi = \frac{\del w}{\del x} \in TR_i.
$$
Therefore $\left(\frac{\del w}{\del t}\right)^\pi \in NR_i$. Since $R_\lambda \in NR_i$,
this proves
$$
\frac{\del w}{\del y} = \left(\frac{\del w}{\del y}\right)^\pi
+ \lambda\left(\frac{\del w}{\del y} \right) R_\lambda
$$
is contained in $NR_i$. This finishes the proof.
\end{proof}

\begin{rem}
We recall that contact triad connection preserves the metric but may have
nonzero torsion, i.e., is not the Levi-Civita connection of the triad metric.
The definition of the second fundamental form of a submanifold $S \subset (M,g)$ for
such a Riemannian connection is still a bilinear map
$
B: TS \times TS \to NS
$
defined by the symmetric average
\be\label{eq:B}
B(X_1,X_2) = \frac12((\nabla_{X_1} X_2)^\perp + (\nabla_{X_2}X_1)^\perp).
\ee
\end{rem}

Since $\xi \perp R_\lambda$ with respect to the triad metric and $\frac{\del w}{\del x} \in TR \subset \xi$ on $\del \dot \Sigma$,
$$
\frac{\del w}{\del x} = \left(\frac{\del w}{\del x}\right)^\pi = - J \left(\frac{\del w}{\del y}\right)^\pi,
$$
and
$$
\frac{\del w}{\del y} = \left(\frac{\del w}{\del y}\right)^\pi + \lambda\left(\frac{\del w}{\del y}\right)
R_\lambda.
$$
Therefore we can express the harmonic energy density function as
$$
e := \left|\frac{\del w}{\del x}\right|^2 + \left|\frac{\del w}{\del y}\right|^2
= 2 \left|\frac{\del w}{\del x}\right|^2 + \left|\lambda\left(\frac{\del w}{\del y}\right)\right|^2.
$$
We then compute
\be\label{eq:*de}
*de|_{\del D} = -\frac{de}{dt} = -4 \left\langle \nabla_t \frac{\del w}{\del x},\frac{\del w}{\del x} \right\rangle
- 2 \frac{\del}{\del y}\left(\lambda\left(\frac{\del w}{\del y}\right)\right)
\cdot \lambda\left(\frac{\del w}{\del y}\right).
\ee
Since $w$ satisfies Neunmann boundary condition, the first term becomes
\be\label{eq:-B}
- 4\left \langle B\left(\frac{\del w}{\del x},\frac{\del w}{\del x}\right),\frac{\del w}{\del y}\right \rangle
\ee
with the second fundamental form \eqref{eq:B} by definition.

For the second term of \eqref{eq:*de}, we prove the following vanishing.
\begin{lem} We have
\be\label{eq:dtdtdw}
 \frac{\del}{\del y}\left(\lambda\left(\frac{\del w}{\del y}\right)\right) = 0
 \ee
 on $\del D_2$.
 In particular, the second term of \eqref{eq:*de} vanishes on $\del D_2$.
\end{lem}
\begin{proof} \emph{By the closedness $d(w^*\lambda \circ j) = 0$}, we obtain
$$
 \frac{\del}{\del y}\left(\lambda\left(\frac{\del w}{\del y}\right)\right)
 +  \frac{\del}{\del x}\left(\lambda\left(\frac{\del w}{\del x}\right)\right) = 0
 $$
  by evaluating $0 = d(w^*\lambda \circ j)(\del_x,\del_y)$.
 Therefore we have
 $$
  \frac{\del}{\del y}\left(\lambda\left(\frac{\del w}{\del y}\right)\right)
 = - \frac{\del}{\del x}\left(\lambda\left(\frac{\del w}{\del x}\right)\right).
 $$
 Since $\frac{\del w}{\del x}$ is tangent to the Legendrian
 $R$ on $\del \dot \Sigma$, the latter vanishes along $\del \dot \Sigma$.
 This finishes the proof.
 \end{proof}

We summarize the above calculation
concerning $*de|_{\del D}$ into the following proposition.

\begin{prop}\label{prop:*de|delD} We have
\be\label{eq:*de-on-bdy}
* de = - 4 \left \langle B\left(\frac{\del w}{\del x},\frac{\del w}{\del x}\right),
\frac{\del w}{\del y}\right \rangle
\ee
on $\del D$.
\end{prop}

Then, substituting \eqref{eq:second-term} and \eqref{eq:*de-on-bdy} into the integral of \eqref{eq:chi2Deltae},
we derive
$$
\left|2\int_{D}\chi^2\Delta e\right|  \leq
\int_D\frac{2\chi^2}{\epsilon}|\nabla(dw)|^2
+ 2\epsilon \|d\chi\|_{C^0(D)}\int_{D_2}|dw|^2 + \int_{\del D} |C(\del D)|
$$
where
$$
C(\del D): = 8
\left \langle B\left(\frac{\del w}{\del x},
 \frac{\del w}{\del x}\right),\frac{\del w}{\del y}\right \rangle
$$
for any positive constant $\epsilon > 0$.

Substituting this into \eqref{eq:intchi2} followed by rearranging the summands, we derive
\beastar
\int_{D}\chi^2|\nabla(dw)|^2
&\leq& \int_D\frac{2\chi^2}{\epsilon}|\nabla(dw)|^2\\
&{}&+\left(4\|K\|_{L^\infty(\dot\Sigma)}+2\epsilon \|d\chi\|_{C^0(D)}\right)\int_{D_2}|dw|^2\\
&{}&+C_1\int_{D_2}|dw|^4 + \int_{\del D} | C(\del D)|
\eeastar
which is equivalent to
\beastar
\int_D \left(1 - \frac{2}{\epsilon}\right) \chi^2 |\nabla(dw)|^2
&\leq &\left(4\|K\|_{L^\infty(\dot\Sigma)}+2\epsilon \|d\chi\|_{C^0(D)}\right)\int_{D_2}|dw|^2\\
&{}&+C_1\int_{D_2}|dw|^4 + \int_{\del D}|C(\del D)|.
\eeastar
We now take $\epsilon=4$. Then we obtain
\beastar
&{}&\int_{D}\chi^2|\nabla(dw)|^2\nonumber\\
&\leq&\left(8\|K\|_{L^\infty(\dot\Sigma)}+16\|d\chi\|^2_{C^0(D)}\right)\int_{D_2}|dw|^2
+2C_1\int_{D_2}|dw|^4 + 2 \int_{\del D} |C(\del D)|.
\eeastar
We recall $D_1 \subset D \subset D_2$, and that $\int_{D_1}|\nabla(dw)|^2 \leq \int_{D}\chi^2|\nabla(dw)|^2$
by the defining properties of $\chi$. Therefore by
letting $D \to D_2$ and setting
\beastar
C_1' & = & C_1'(D_1,D_2) =  8\|K\|_{L^\infty(\dot\Sigma)}+16\|d\chi\|^2_{C^0(D)}\\
C_2' & = & C_2'(D_1,D_2) = 2C_1
\eeastar
 we have obtained
$$
\int_{D_1}|\nabla(dw)|^2 \leq C_1' \int_{D_2} |dw|^2 + C_2' \int_{D_2} |dw|^4
+ 2\int_{\del D} |C(\del D)|.
$$
By removing the primes and adjusting the constants,  we have finally finished
the proof of \eqref{eq:differential-inequality} and
hence Theorem \ref{thm:local-W12}.
\end{proof}

\section{$C^{k,\delta}$ coercive estimates for $k \geq 1$}
\label{sec:Ckdelta-estimates}

Once we have established $W^{2,2}$ estimate, we could proceed with the $W^{k+2,2}$ estimate $k \geq 1$
inductively as in \cite[Section 5.2]{oh-wang:CR-map1}.
The effect of the Legendrian boundary condition on
the higher derivative estimate is not quite straightforward, although it should be doable.
Instead we take an easier path of using the H\"older estimates instead of
the Sobolev estimates by expressing the following fundamental equation in
the isothermal coordinates of $(\dot \Sigma,j)$. The current approach also simplifies the higher derivative
estimate given in \cite[Section 5.2]{oh-wang:CR-map1}
which do the $W^{k+2,2}$-estimates instead.

\begin{thm}[Theorem 4.2, \cite{oh-wang:connection}]
Let $w$ satisfy $\delbar^\pi w=0$. Then
\be\label{eq:fundamental}
d^{\nabla^\pi}(d^\pi w) = -w^*\lambda\circ j \wedge\left( \frac{1}{2}(\CL_{R_\lambda}J)\, d^\pi w\right).
\ee
\end{thm}
(For readers' convenience, we recall the definition of the covariant exterior differential
$d^{\nabla^\pi}$ and other relevant calculus for vector valued forms in general
in Appendix \ref{append:weitzenbock}.)

In an isothermal coordinates $z = x+iy$, with the evaluation of $(\frac{\del}{\del x},\frac{\del}{\del y})$
into \eqref{eq:fundamental}, the equation becomes
$$
\nabla_x^\pi \zeta + J \nabla_y^\pi \zeta
+ \frac{1}{2} \lambda\left(\frac{\del w}{\del y}\right)(\CL_{R_\lambda}J)\zeta - \frac{1}{2}
\lambda\left(\frac{\del w}{\del x}\right)(\CL_{R_\lambda}J)J\zeta =0.
$$
(See \cite[Corollary 4.3]{oh-wang:connection}.) By writing
$$
\overline \nabla^\pi := \nabla^{\pi(0,1)} = \frac{\nabla^\pi + J \nabla^\pi_{j(\cdot)}}{2}
$$
which is the anti-complex linear part of $\nabla^\pi$, and the linear operator
$$
P_{w^*\lambda}(\zeta): = \frac{1}{4} \lambda\left(\frac{\del w}{\del y}\right)(\CL_{R_\lambda}J)\zeta - \frac{1}{4}
\lambda\left(\frac{\del w}{\del x}\right)(\CL_{R_\lambda}J)J\zeta,
$$
the equation becomes
\be\label{eq:nablabar-P}
\overline \nabla^\pi \zeta + P_{w^*\lambda}(\zeta) = 0
\ee
which is a linear first-order PDE of Cauchy-Riemann type. We note that
by the Sobolev embedding, $ W^{2,2} \subset C^{0,\delta}$ for $0 \leq \delta < 1/2$.
Therefore we start from $C^{0,\delta}$ bound with $0 < \delta <1/2$ and will inductively bootstrap it to
$C^{k, \delta}$ bounds for $k \geq 1$

WLOG, we assume that $D_2 \subset \dot \Sigma$ is a semi-disc with $\del D \subset \del \dot \Sigma$
and equipped with an isothermal coordinates $(x,y)$ such that
$$
D_2 = \{ (x,y) \mid |x|^2 + |y|^2 < \epsilon, \, y \geq 0\}
$$
for some $\epsilon > 0$
and so $\del D_2 \subset \{(x,y) \in D \mid y = 0\}$. Assume $D_1 \subset D_2$
is the semi-disc with radius $\epsilon /2$.
We denote $\zeta = \pi \frac{\del w}{\del x}$, $\eta = \pi \frac{\del w}{\del y}$
as in \cite{oh-wang:CR-map1}, and consider the complex-valued function
\be\label{eq:alpha}
\alpha(x,y) = \lambda\left(\frac{\del w}{\del y}\right)
+ \sqrt{-1}\left(\lambda\left(\frac{\del w}{\del x}\right)\right)
\ee
as in \cite[Subsection 11.5]{oh-wang:CR-map2}. We note that since $w$ satisfies the
Legendrian boundary condition, we have
\be\label{eq:lambda(delw)=0}
\lambda\left(\frac{\del w}{\del x}\right) = 0
\ee
on $\del D_2$.

\begin{lem}[Lemma 11.19 \cite{oh-wang:CR-map2}]\label{lem:*dw*lambda}
 Let $\zeta = \pi \frac{\del w}{\del x}$. Then
$$
*d(w^*\lambda) =|\zeta|^2.
$$
\end{lem}
Combining Lemma \ref{lem:*dw*lambda} together with the equation $d(w^*\lambda\circ j)=0$, we
notice that $\alpha$ satisfies the equations
\be\label{eq:atatau-equation}
\begin{cases}
\delbar \alpha =\nu, \quad \nu =\frac{1}{2}|\zeta|^2 \\
\alpha(z) \in \R \quad z  \quad \text{\rm for } \, \in \del D_2
\end{cases}
\ee
thanks to \eqref{eq:lambda(delw)=0},
where $\delbar=\frac{1}{2}\left(\frac{\del}{\del x}+\sqrt{-1}\frac{\del}{\del y}\right)$
is the standard Cauchy-Riemann operator for the standard complex structure $J_0=\sqrt{-1}$.

Then we arrive at the following system of equations for the pair $(\zeta,\alpha)$
\be\label{eq:equation-for-zeta0}
\begin{cases}\nabla_x^\pi \zeta + J \nabla_y^\pi \zeta
+ \frac{1}{2} \lambda(\frac{\del w}{\del y})(\CL_{R_\lambda}J)\zeta - \frac{1}{2} \lambda(\frac{\del w}{\del x})(\CL_{R_\lambda}J)J\zeta =0\\
\zeta(z) \in TR_i \quad \text{for } \, z \in \del D_2
\end{cases}
\ee
for some $i = 0, \ldots, k$, and
\be\label{eq:equation-for-alpha}
\begin{cases}
\delbar \alpha = \frac{1}{2}|\zeta|^2 \\
\alpha(z) \in \R \quad \text{for } \, z \in \del D_2.
\end{cases}
\ee
These two equations form a nonlinear elliptic system for $(\zeta,\alpha)$ which are coupled:
$\alpha$ is fed into
\eqref{eq:equation-for-zeta0} through its coefficients and then $\zeta$ provides the input
for the equation \eqref{eq:equation-for-alpha} and then back and forth. Using this structure of
coupling, we are now ready to derive the higher derivative estimates
by alternating boot strap arguments between $\zeta$ and $\alpha$
which is now in order.

\begin{thm}\label{thm:local-regularity}
Let $w$ be a contact instanton satisfying \eqref{eq:contacton-Legendrian-bdy-intro}.
Then for any pair of domains $D_1 \subset D_2 \subset \dot \Sigma$ such that $\overline{D_1}\subset D_2$, we have
$$
\|dw\|_{C^{k,\delta}(D_1)} \leq C_\delta( \|dw\|_{W^{1,2}(D_2)})
$$
for some  positive function $C_\delta = C_\delta(r)$ that
is continuous at $r = 0$ which depends on $J$, $\lambda$ and $D_1, \, D_2$
but independent of $w$.
\end{thm}

The rest of this section will be occupied by the proof which
employ the arguments of the alternating boot strap between $\zeta$ and $\alpha$
back and forth.

\subsection{Start of alternating boot-strap: $W^{1,2}$-estimate for $dw$.}

By writing the Reeb component $w^*\lambda = f \, dx + g\, dy$
in the isothermal coordinate $(x,y)$ with
$$
f = \lambda \left(\frac{\del w}{\del x}\right) , \quad g = \lambda \left(\frac{\del w}{\del y}\right)
$$
we express the complex-valued function $\alpha$ as
$\alpha = g + \sqrt{-1}f$ defined above in \eqref{eq:alpha}. Then
$\zeta$ satisfies the equation \eqref{eq:equation-for-zeta0}.
It is obvious to see that when $w^*\lambda = f \, dx + g\, dy$ is given
\eqref{eq:equation-for-zeta0} is a
linear elliptic equation of the Cauchy-Riemann type
with totally real boundary equation for $\zeta$.

Then Theorem  \ref{thm:local-W12} is directly translated into the following
$W^{1,2}$ bound for $\zeta$.

\begin{lem}\label{lem:W12-zeta} We have
$$
\|\zeta\|_{W^{1,2}(D_1)}^2 + \|w^*\lambda\|_{W^{1,2}(D_1)}^2
 \leq  C_4 (\|\zeta\|_{L^4(D_2)} ^4+ \|w^*\lambda\|_{L^4(D_1)}^4)
 $$
 with some adjustment of the constant $C_4$.
 In particular we have
 \bea
 \|\zeta\|_{W^{1,2}(D_1)}^2
&  \leq  & C_4 (\|\zeta\|_{L^4(D_2)}^4 + \|w^*\lambda\|_{L^4(D_1)}^4)\\
 \|w^*\lambda\|_{W^{1,2}(D_1)} \label{eq:W22-zeta}
&  \leq  & C_4 (\|\zeta\|_{L^4(D_2)}^4 + \|w^*\lambda\|_{L^4(D_1)}^4).
\label{eq:W22-fg}
 \eea
\end{lem}

We then consider the equation \eqref{eq:equation-for-zeta0} for $\zeta$ and
 obtain the estimate
$$
\|\overline \nabla^\pi \zeta\|_{W^{1,2}(D_1)}^2 \leq C (\|f\|_{W^{1,4}(D_2)}
+ \|g\|_{W^{1,4}(D_2)})\|\zeta\|_{W^{1,4}(D_2)}.
$$
By the standard estimate for the Riemann-Hilbert problem
with Dirichlet boundary condition for the Cauchy-Riemann operator
$\overline \nabla^\pi$ in $W^{2,2}$ Sobolev space combined with the inequality
$\|\zeta\|_{C^{0,\delta}(D_1)} \leq C \|\zeta\|_{W^{2,2}(D_1)}$, we derive
$$
\|\zeta\|_{C^{0,\delta}(D_1)}^2 \leq C_1(\|f\|_{W^{1,4}(D_2)} + \|g\|_{W^{1,4}(D_2)})
\|\zeta\|_{W^{1,4}(D_2)}
$$
with say $\delta = 2/5 < 1/2$. We attract readers' attention that the right hand side
involves only the $W^{1,4}$ norm of $dw = d^\pi w + w^*\lambda\, R_\lambda$.

\subsection{$C^{1,\delta}$-estimate for $w^*\lambda = f\, dx + g\, dy$}

Next we consider \eqref{eq:equation-for-alpha} which is another
Riemann-Hilbert problem with the real (or imaginary) boundary condition for the
Cauchy-Riemann operator $\delbar$.
Then again by the standard estimate for the Riemann-Hilbert problem
with the real (or imaginary) boundary condition, we derive
\be\label{eq:|alpha|22}
\|\alpha\|_{C^{1,\delta}(D_1)} \leq C_2 (\|\zeta\|^2_{C^{0,\delta}(D_2)}, \|\alpha\|_{C^0(D_2)}).
\ee
(See the basic H\"older estimate from \cite[Proposition 2.3.6]{sikorav:holo}
 for the equation of the type
$$
\delbar f + q(f) \del f = 0
$$
for example.)

\subsection{$C^{1,\delta}$-estimate for $d^\pi w$.}

At this stage, \eqref{eq:equation-for-zeta0}, which is of the form \eqref{eq:nablabar-P},
 implies that there exists a positive function $C_\delta = C_\delta(r)$  such that
$$
\|\overline \nabla^\pi \zeta\|_{C^{0,\delta}(D_1)} \leq C_\delta (\|\zeta \|_{C^{0,\delta}(D_2)})
$$
for all solutions thereof.  Here the function $C_\delta(r)$ is defined for all sufficiently small
$r$ (and continuous at $r = 0$) whose smallness depends on the bound for
$\|\zeta \|_{C_\delta^{0,\delta}(D_2)}$ obtained in the previous stage.

Again by the estimate for the equation of the Cauchy-Riemann type
$\delbar f + q(f) \del f = 0$  applied to \eqref{eq:equation-for-zeta0},
any solution $\zeta$ thereof indeed satisfies
$$
\|\zeta\|_{C^{1,\delta}(D_1)} \leq  C_\delta (\|\zeta \|_{C^{0,\delta}(D_2)},
\|\zeta \|_{C^{0}(D_2)})
$$
where again $C_\delta(r,s)$ is a continuous function at $(r,s) = (0,0)$.
Combining the two, we have derived
\be\label{eq:zeta-1delta}
\|\zeta\|_{C^{1,\delta}(D_1)}\leq C_\delta( \|\zeta \|_{C^{0,\delta}(D_2)},
\|\zeta \|_{C^{0}(D_2)})
\ee
for all solutions of \eqref{eq:equation-for-zeta}.

\subsection{$C^{2,\delta}$-estimate for $w^*\lambda$}

Now we consider \eqref{eq:equation-for-alpha} back.
By differentiating it by $x$, we get the same-type of equation
$$
\begin{cases}
\delbar (\nabla_x \alpha) + [\nabla_x,\delbar] \alpha = \nabla_x \nu \\
\nabla_x \alpha \in \R, \quad \text{\rm for }\,  z \in \del D_2
\end{cases}
$$
where $[\nabla_x,\delbar]$ is the commutator of $\nabla_x$ and $\delbar$ which is
a zero-order operator.
We have
$$
\nabla_x \nu = \langle \nabla_x \zeta, \zeta \rangle
$$
and hence the bound
$$
\|\nabla_\tau \nu\|_{C^{0,\delta}} \leq \|\zeta\|_{C^0} \|\nabla_x  \zeta\|_{C^{0,\delta}}
\|\zeta\|_{C^{0,\delta}} \|\nabla_x  \zeta\|_{C^0} \leq
C_\delta(\|\zeta_{C^{0,\delta}}, \|\zeta\|_{C^0}).
$$
By repeating the $C^{1,\delta}$-estimate from \cite[Proposition 2.3.6 (ii)]{sikorav:holo}
applied to $\nabla_x \alpha$ instead of $\alpha$, we obtain
\be\label{eq:alpha-step2}
\|\nabla_\tau \alpha \|_{C^{1,\delta}(D_1)}\leq C_\delta( \|\zeta \|_{C^{0,\delta}(D_2)},
\|\alpha\|_{C^{0,\delta}(D_2)}).
\ee
Then noting that \eqref{eq:equation-for-alpha} can be written as
$$
\frac12 \left( \frac{\del \alpha}{\del x} + i \frac{\del \alpha}{\del y}\right) = \frac12 \|\zeta\|^2,
$$
\eqref{eq:alpha-step2} also implies similar estimate for $\nabla_y \alpha = \frac{\del \alpha}{\del y}$.
This implies the inequality
$$
\|\nabla \alpha\|_{C^{1,\delta}(D_1)}\leq C_\delta( \|\zeta \|_{C^{0,\delta}(D_2)})
$$
which proves the $C^{2,\delta}$-estimate for $\alpha$.

\subsection{$C^{2,\delta}$-estimate for $d^\pi w$}

Then we go back to \eqref{eq:equation-for-zeta}.
In this step, we temporarily abandon the usage of triad connection but take
the Levi-Civita connection $\nabla^{\text \rm{LC}}$
with respect to which $R_i$ are totally geodesic as mentioned before
to take the covariant differential. (This is because if we take $\nabla_x\zeta$, the latter
may not be tangent to $R_i$ and so we cannot apply the previous step of \eqref{eq:equation-for-zeta0}.)

Using the fact that $\nabla^{\text{\rm LC}} = \nabla + A$ with a zero-order operator
$A$,  it will give rise to
\be\label{eq:equation-for-zeta}
\begin{cases}
\nabla_x^\pi (\nabla_x^{\text{\rm LC}} \zeta)
+ J \nabla_y^\pi  (\nabla_x^{\text{\rm LC}} \zeta)
- \frac12  g\, \CL_{R_\lambda}J \, (\nabla_x^{\text{\rm LC}} \zeta) + \frac12 f (\CL_{R_\lambda}J) J (\nabla_x^{\text{\rm LC}} \zeta)
+ Q(\zeta)= 0,\\
\nabla_x^{\text{\rm LC}} \zeta \in TR_i \quad \text{\rm for } z \in \del D_2
\end{cases}
\ee
where $Q$ is a linear zero-order operator of $\zeta$ whose coefficients
depend at most on $\zeta\, \, f, \, \nabla_x f, \, g, \, \nabla_x g$ which are all
of the class $C^{1,\delta}$ by now. The equation is again a Cauchy-Riemann type
for the variable $\nabla^{\text{\rm LC}}\zeta$.

The term $Q(\zeta)$ appears because taking the covariant
differential of \eqref{eq:equation-for-zeta0} not only involves $\nabla^{\text{\rm LC}}_x \zeta$
but also the differential of the coefficient functions of the linear equation, and
the commutator of the
$$
\left [\nabla^\pi_x,\nabla_x^{\text{\rm LC}} \right].
$$
Therefore by applying the previous $C^{1,\delta}$-estimate to $\nabla_x^{\text{\rm LC}}\zeta$
instead of $\zeta$, we obtain
$$
\|\nabla^{\text{\rm LC}}_x \zeta\|_{C^{1,\delta}(D_1)} \leq
C_\delta(\|\zeta\|_{C^{0,\delta}(D_2)}, \|\alpha \|_{C^{1,\delta}(D_2)},
\|\zeta\|_{C^0(D_2)}).
$$
Combining the above, we  then derive
$$
\|\zeta\|_{C^{2,\delta(D_1)}} \leq C_\delta(\|\alpha\|_{C^{1,\delta}(D_2)},
\|\zeta\|_{C^{1,\delta}(D_2)},
\|\zeta\|_{C^0(D_2)}, \|\alpha\|_{C^0(D_2)})
$$
with $C_\delta$ is the same kind of function as above.

\subsection{Wrap-up of the alternating boot-strap argument}

Now we repeat the above alternating boot strap arguments between $\zeta$ and $\alpha$
back and forth by taking the differential with respect to $\nabla_x^{\text{\rm LC}}$
to inductively derive the $C^{k,\delta}$-estimates both for $\zeta$ and $\alpha$
in terms of $\|\zeta\|_{L^4(D_2)}$ and $\|\alpha\|_{L^4(D_2)}$ which is equivalent to
considering the full $\|dw\|_{L^4}$.
This completes the proof of Theorem \ref{thm:local-regularity}.

\begin{rem} The polynomial growth of the number of arguments appearing in
the function $C_\delta$ above as taking derivatives more and more
 is consistent with similar polynomial  bound of higher covariant differential $(\nabla)^{k+1}(dw)$
 in terms of  $d^\pi w$ and $w^*\lambda$ in their $W^{k+2,2}$ Sobolev space bound
given for the contact instanton in the closed string context
in \cite[Theorem 1.7]{oh-wang:CR-map1}.
\end{rem}

\section{Vanishing of asymptotic charge and subsequence convergence}
\label{sec:subsequence-convergence}

In this section, we study the asymptotic behavior of contact instantons
on the Riemann surface $(\dot\Sigma, j)$ associated with a metric $h$ with \emph{strip-like ends}.
To be precise, we assume there exists a compact set $K_\Sigma\subset \dot\Sigma$,
such that $\dot\Sigma-\Int(K_\Sigma)$ is a disjoint union of punctured semi-disks
 each of which is isometric to the half strip $[0, \infty)\times [0,1]$ or $(-\infty, 0]\times [0,1]$, where
the choice of positive or negative strips depends on the choice of analytic coordinates
at the punctures.
We denote by $\{p^+_i\}_{i=1, \cdots, l^+}$ the positive punctures, and by $\{p^-_j\}_{j=1, \cdots, l^-}$ the negative punctures.
Here $l=l^++l^-$. Denote by $\phi^{\pm}_i$ such strip-like coordinates.
We first state our assumptions for the study of the behavior of boundary punctures.
(The case of interior punctures is treated in \cite[Section 6]{oh-wang:CR-map1}.)

\begin{defn}Let $\dot\Sigma$ be a boundary-punctured Riemann surface of genus zero with punctures
$\{p^+_i\}_{i=1, \cdots, l^+}\cup \{p^-_j\}_{j=1, \cdots, l^-}$ equipped
with a metric $h$ with \emph{strip-like ends} outside a compact subset $K_\Sigma$.
Let
$w: \dot \Sigma \to M$ be any smooth map with Legendrian boundary condition.
We define the total $\pi$-harmonic energy $E^\pi(w)$
by
\be\label{eq:endenergy}
E^\pi(w) = E^\pi_{(\lambda,J;\dot\Sigma,h)}(w) = \frac{1}{2} \int_{\dot \Sigma} |d^\pi w|^2
\ee
where the norm is taken in terms of the given metric $h$ on $\dot \Sigma$ and the triad metric on $M$.
\end{defn}

We put the following hypotheses in our asymptotic study of the finite
energy contact instanton maps $w$ as in \cite{oh-wang:CR-map1}, \emph{except not requiring the charge vanishing
condition $Q = 0$, which itself we will prove here under the hypothesis using the Legendrian boundary condition}:

\begin{hypo}\label{hypo:basic}
Let $h$ be the metric on $\dot \Sigma$ given above.
Assume $w:\dot\Sigma\to M$ satisfies the contact instanton equation \eqref{eq:contacton-Legendrian-bdy-intro}
and
\begin{enumerate}
\item $E^\pi_{(\lambda,J;\dot\Sigma,h)}(w)<\infty$ (finite $\pi$-energy);
\item $\|d w\|_{C^0(\dot\Sigma)} <\infty$.
\item $\Image w \subset \mathsf K \subset M$ for some compact set $\mathsf K$.
\end{enumerate}
\end{hypo}

Throughout this section, we work locally near one boundary puncture $p$, i.e., on a punctured semi-disc
$D^\delta(p) \setminus \{p\}$. By taking the associated conformal coordinates $\phi^+ = (\tau,t)
:D^\delta(p) \setminus \{p\} \to [0, \infty)\times [0,1]$ such that $h = d\tau^2 + dt^2$,
we need only look at a map $w$ defined on the half strip $[0, \infty)\times [0,1]\to M$
without loss of generality.

The above finite $\pi$-energy and $C^0$ bound hypotheses imply
\be\label{eq:hypo-basic-pt}
\int_{[0, \infty)\times [0,1]}|d^\pi w|^2 \, d\tau \, dt <\infty, \quad \|d w\|_{C^0([0, \infty)\times [0,1])}<\infty
\ee
in these coordinates.

We will use the following equality which is derived in \cite{oh-wang:connection}. (See \cite[Equation (3.1)]{oh-wang:connection}.)
\begin{lem} Let $h$ be a K\"aher metric of $(\Sigma,j)$. Suppose $w$ satisfies $\delbar^\pi w = 0$. Then
\be\label{eq:dw*lambda}
d(w^*\lambda) = \frac12 |d^\pi w|^2 dA
\ee
where $dA$ is the area form of $h$.
\end{lem}

Let $w$ satisfy Hypothesis \ref{hypo:basic}.
In addition we assume that the limit $\lim_{\tau \to \infty} w(\tau, \dot) \to \infty$ exists. Then
we can associate two
natural asymptotic invariants at each puncture defined as
\beastar
T & := & \lim_{r \to \infty} \int_{\{r\}\times [0,1]} (w|_{\{0\}\times [0,1] })^*\lambda \label{eq:TQ-T}\\
Q & : = & \lim_{r \to \infty} \int_{\{r\}\times [0,1]}((w|_{\{0\}\times [0,1] })^*\lambda\circ j).\label{eq:TQ-Q}
\eeastar
(Here we only look at positive punctures. The case of negative punctures is similar.)
We call $T$ the \emph{asymptotic contact action}
and $Q$ the \emph{asymptotic contact charge} of the contact instanton $w$ at the given puncture.

\begin{rem}\label{rem:TQ}
In particular \eqref{eq:dw*lambda} holds for any contact instanton $w$.  By Stokes' formula, we can express
$$
T = \frac{1}{2} \int_{[s,\infty) \times [0,1]} |d^\pi w|^2 + \int_{\{s\}\times [0,1]}(w|_{\{s\}\times [0,1]})^*\lambda, \quad
\text{for any } s\geq 0
$$
does not depend on $s \geq 0$.
\end{rem}

The proof of the following subsequence convergence result largely follows that of
\cite[Theorem 6.4]{oh-wang:CR-map1}. Since we need to take care of the Legendrian boundary condition and also need to prove that the charge $Q$ vanishes in the course of the proof,
we duplicate the details of the proof therein in the first half of our proof and then
explain in the second half how the charge vanishing occurs
under the Legendrian boundary condition. One may say that \emph{the presence of Legendrian barrier
prevents the instanton from spiraling.}

\begin{thm}[Subsequence Convergence]\label{thm:subsequence}
Let $w:[0, \infty)\times [0,1]\to M$ satisfy the contact instanton equations \eqref{eq:contacton-Legendrian-bdy-intro}
and Hypothesis \eqref{eq:hypo-basic-pt}.

Then for any sequence $s_k\to \infty$, there exists a subsequence, still denoted by $s_k$, and a
massless instanton $w_\infty(\tau,t)$ (i.e., $E^\pi(w_\infty) = 0$)
on the cylinder $\R \times [0,1]$  that satisfies the following:
\begin{enumerate}
\item $\delbar^\pi w_\infty = 0$ and
$$
\lim_{k\to \infty}w(s_k + \tau, t) = w_\infty(\tau,t)
$$
in the $C^l(K \times [0,1], M)$ sense for any $l$, where $K\subset [0,\infty)$ is an arbitrary compact set.
\item $w_\infty$ has vanishing asymptotic charge $Q = 0$ and satisfies $w_\infty(\tau,t)= \gamma(T\, t)$
for some Reeb chord $\gamma$ is some Reeb chord joining $R_0$ and $R_1$ with period $T$ at each puncture.
\item $T \neq 0$ at each  puncture with the associated pair $(R,R')$ of boundary condition with $R \cap R'
= \emptyset$.
\end{enumerate}
\end{thm}
\begin{proof}
For a given contact instanton $w: [0, \infty)\times [0,1]\to M$, we consider the family of maps
$w_s: [-s, \infty) \times [0,1] \to M$ defined by
$w_s(\tau, t) = w(\tau + s, t)$ with Legendrian boundary condition
$$
w_s(\tau,0) \in R, \quad w_s(\tau,1) \in R'
$$
for $R,\, R' \in \{R_0, \ldots, R_k\}$.
For any compact set $K\subset \R$, there exists some $s_0$ large enough such that
$K\subset [-s, \infty)$ for every $s\geq s_0$. For such $s\geq s_0$, we can also get an $[s_0, \infty)$-family of maps by defining $w^K_s:=w_s|_{K\times [0,1]}:K\times [0,1]\to M$.

The asymptotic behavior of $w$ at infinity can be understood by studying the limiting behavior of the sequence of maps
$\{w^K_s:K\times [0,1]\to M\}_{s\in [s_0, \infty)}$, for each given compact set $K\subset \R$.

First of all,
it is easy to check that under Hypothesis \ref{hypo:basic}, the family
$\{w^K_s:K\times [0,1]\to M\}_{s\in [s_0, \infty)}$ satisfies the following
\begin{enumerate}
\item
For every $s\in [s_0, \infty)$,
$$
\begin{cases}
\delbar^\pi w^K_s=0, \quad d((w^K_s)^*\lambda\circ j)=0 \\
w_s(\tau,0) \in R, \quad w_s(\tau,1) \in R'.
\end{cases}
$$
\item $\lim_{s\to \infty}\|d^\pi w^K_s\|_{L^2(K\times [0,1])}=0$
\item $\|d w^K_s\|_{C^0(K\times [0,1])}\leq \|d w\|_{C^0([0, \infty)\times [0,1])}<\infty$.
\end{enumerate}

From (1) and (3) together with the compactness of $\overline{\Image w} \subset M$
(which provides a uniform $L^2(K\times [0,1])$ bound)
and Theorem \ref{thm:local-regularity}, we obtain
$$
\|w^K_s\|_{C^{1,\alpha}(K\times [0,1])}\leq C_{K;(1,\alpha)}<\infty,
$$
for some constant $C_{K;(1,\alpha)}$ independent of $s$.
Then
 $\{w^K_s:K\times [0,1]\to M\}_{s\in [s_0, \infty)}$ is sequentially pre-compact.
Therefore, for any sequence $s_k \to \infty$, there exists a subsequence, still denoted by $s_k$,
and some limit $w^K_\infty\in C^{1}(K\times [0,1], M)$ (which may depend on the subsequence $\{s_k\}$), such that
$$
w^K_{s_k}\to w^K_{\infty}, \quad \text {as } k\to \infty,
$$
in the $C^{1}(K\times [0,1], M)$-norm sense.
Furthermore, combining this with (2) and \eqref{eq:dw*lambda}, we get
$$
dw^K_{s_k}\to dw^K_{\infty} \quad \text{and} \quad dw^K_\infty=(w^K_\infty)^*\lambda\otimes R_\lambda,
$$
and both $(w^K_\infty)^*\lambda$ and $(w^K_\infty)^*\lambda\circ j$ are harmonic one-forms: Closedness of
$(w^K_\infty)^*\lambda$ follows from \eqref{eq:dw*lambda} and the convergence
$$
|d^\pi w^K_\infty|^2 = \lim_{k \to \infty}|d^\pi w^K_{s_k}|^2 = 0.
$$
Closedness of $(w^K_\infty)^*\lambda\circ j$, which is equivalent to the coclosedness of $(w^K_\infty)^*\lambda$,
follows from the hypothesis. The limit map $w_\infty^K$ also satisfies
Legendrian boundary conditions
$$
w_\infty^K(\tau,0) \in R, \quad w_\infty^K(\tau,1) \in R'
$$
since $w^K_{s_k} \to w^K_\infty$ in $C^1$-topology and each $w^K_{s_k}$ satisfies
the same Legendrian boundary condition.

Note that these limiting maps $w^K_\infty$ have a common extension $w_\infty: \R\times [0,1]\to M$
by a  diagonal sequence argument where, one takes a sequence of compact sets $K$ one including another and exhausting $\R$.
Then it follows that  $w_\infty$ is $C^1$,  satisfies
$$
\|d w_\infty\|_{C^0(\R\times [0,1])}\leq \|d w\|_{C^0([0, \infty)\times [0,1])}<\infty
$$
and $d^\pi w_\infty = 0$. Therefore
$$
dw_\infty=w_\infty^*\lambda \otimes R_\lambda
$$
Then we derive from Theorem \ref{thm:local-regularity} that $w_\infty$ is actually in $C^\infty$.
Also notice that both $w_\infty^*\lambda$ and $w_\infty^*\lambda\circ j$ are bounded harmonic one-forms on $\R\times [0,1]$.
\begin{lem} We have
$$
w_\infty^*\lambda= b\,dt, \quad w_\infty^*\lambda\circ j=b\,d\tau
$$
for some constant $b$.
\end{lem}
\begin{proof}
We have $w_\infty^*\lambda = f\,d\tau + g \, dt$ for some bounded functions $f, \, g$.
Furthermore $w_\infty^*\lambda$ is harmonic if and only if $(f,g)$ satisfies the Cauchy-Riemann equation
\be\label{eq:harmonic-fg}
\frac{\del f}{\del t} = \frac{\del g}{\del \tau}, \, \frac{\del f}{\del \tau} = - \frac{\del g}{\del t}
\ee
for the complex coordinate $\tau + it$.
This in particular implies $\Delta f = 0$. On the other hand,
the Legendrian boundary condition
$$
w_\infty(\tau,0) \in R, \quad w(\tau,1) \in R'
$$
implies
\be\label{eq:Q=0}
\lambda\left(\frac{\del w_\infty}{\del \tau}(\tau,0)\right) = 0 = \lambda\left(\frac{\del w_\infty}{\del \tau}(\tau,1)\right).
\ee
Therefore $f$ is bounded and satisfies
$$
\begin{cases}
\Delta f = 0 \quad &\text{on }\, R \times [0,1]\\
f(\tau,0) = f(\tau,1) = 0 \quad &\text{for all }\, \tau \in \R.
\end{cases}
$$
It follows from standard results on the harmonic functions on the strip $\R \times [0,1]$
satisfying the Dirichlet boundary condition that $f = 0$.
The equation \eqref{eq:harmonic-fg} in turn implies $\frac{\del g}{\del \tau} = 0 = \frac{\del g}{\del t}$
from which we derive $g$ is a constant function.
Therefore the forms $w_\infty^*\lambda$ and $w_\infty^*\lambda\circ j$  must be written in  the form
$$
w_\infty^*\lambda= b\,dt, \quad w_\infty^*\lambda\circ j=b\,d\tau.
$$
for some constants $b$. This finishes the proof.
\end{proof}

Obviously we have
$$
\int_{\{r\}\times [0,1]}(w_\infty|_{\{r\}\times [0,1]})^*\lambda\circ j = \int_{\{r\} \times [0,1]}b\, d\tau = 0
$$
for all $r \geq 0$ and hence
$$
Q =  \lim_{r \to  \infty} \int_{\{r\}\times [0,1]}\lim_{k\to \infty}(w_{s_k}|_{\{r\}\times [0,1]})^*\lambda\circ j
= 0.
$$
We now show that $T = b$.

By taking an arbitrary point $r\in K$, since $w_\infty|_{\{r\}\times [0,1]}$ is the limit of some sequence $w_{s_k}|_{\{r\}\times [0,1]}$
in the $C^1$ sense, we derive
\beastar
b& = & \int_{\{r\}\times [0,1]}(w_\infty|_{\{r\}\times [0,1]})^*\lambda
= \int_{\{r\}\times [0,1]}\lim_{k\to \infty}(w_{s_k}|_{\{r\}\times [0,1]})^*\lambda\\
&=&\lim_{k\to \infty}\int_{\{r\}\times [0,1]}(w_{s_k}|_{\{r\}\times [0,1]})^*\lambda
= \lim_{k\to \infty}\int_{\{r+s_k\}\times [0,1]}(w|_{\{r+s_k\}\times [0,1]})^*\lambda\\
&=&\lim_{k\to \infty}(T-\frac{1}{2}\int_{[r+s_k, \infty)\times [0,1]}|d^\pi w|^2)\\
&=&T-\lim_{k\to \infty}\frac{1}{2}\int_{[r+s_k, \infty)\times [0,1]}|d^\pi w|^2
= T
\eeastar
Here in the derivation, we use Remark \ref{rem:TQ}.
As in \cite{oh-wang:CR-map1}, we conclude that the image of $w_\infty$ is contained in
a single leaf of the Reeb foliation by the connectedness of $\R \times [0,1]$.
Let $\gamma: \R \to M$ be a parametrisation of
the leaf so that $\dot \gamma = R_\lambda(\gamma)$. Then we can write $w_\infty(\tau, t)=\gamma(s(\tau, t))$, where
$s:\R\times [0,1]\to \R$ and $s=T\,t+c_0$ since $ds=T\,dt$, where $c_0$ is some constant.

From this we derive that, if $T\neq 0$,
$\gamma$ is a (nonconstant) Reeb chord of period $T$ joining $R$ and $R'$.
Of course, $T$ could be zero and $w_\infty$ is a constant map and so corresponds to an intersection point $R \cap R'$. The latter case will not occur provided $R \cap R' =\emptyset$. This finishes the proof.
This finishes the proof.
\end{proof}

\begin{rem}\label{rem:big-difference} This charge vanishing $Q = 0$
of the massless instanton is a big advantage for
the current open string case which is a big deviation from
the closed string case studied in \cite{oh-wang:CR-map1,oh-wang:CR-map2,oh:contacton}.
We refer to \cite[Section 8.1]{oh:contacton} for the classification of massless instantons
on the cylinder $\R \times S^1$ for which the kind of
massless instantons with $Q \neq 0$ but $T = 0$ appear
 or the closed string case. This is called the \emph{appearance of spiraling instantons along the
Reeb core}, which is the only obstacle towards the Fredholm theory and the compactification of
the moduli space of finite energy contact instantons for the general closed string context.

The above theorem removes this obstacle for the open string case of Legendrian boundary condition.
One could say that the presence of the Legendrian obstacle
blocks this spiraling phenomenon of the contact instantons.
In sequels to the present paper as in \cite{oh:entanglement1,oh:entanglement2},
we will construct the Fukaya type category on contact manifolds whose objects are Legendrian submanifolds and morphisms and products by counting appropriate
instantons with prescribed asymptotic conditions and Legendrian boundary conditions.
\end{rem}

\begin{rem} We mention that for a generic $k+1$ tuple $\{R_0,\ldots, R_k\}$
 $R_i$'s are pairwise disjoint  by the dimensional reason, and a constant solution
 cannot arise as an asymptotic chord connecting two different $R \neq R' $ (trans-chord)
 with a pair $(R,R')$ therefrom. On the other hand, when $R = R'$, the asymptotic
 limit at the relevant puncture could be a constant curve which will then have $Q = 0 = T$.
 In this case it is shown in \cite[Theorem 8.7]{oh:contacton} that such a puncture is
 removable for the case of interior puncture \emph{under the finite energy hypothesis}:
 The same proof applies to the current case
 of boundary puncture with the a priori boundary estimates provided in Theorem \ref{thm:local-W12}
 and \ref{thm:local-regularity}.
 \end{rem}

From the previous theorem, we  immediately get the following corollary as in \cite[Section 8]{oh-wang:CR-map1}. We first recall the standard notion of nondegeneracy of Reeb chords.

\begin{defn}\label{defn:nondegeneracy-chords}
We say a Reeb chord $(\gamma, T)$ of $(R_0,R_1)$ is nondegenerate if
the linearization map $\Psi_\gamma = d\phi^T(p): \xi_p \to \xi_p$ satisfies
$$
\Psi_\gamma(T_{\gamma(0)} R_0) \pitchfork T_{\gamma(1)} R_1  \quad \text{\rm in }  \,  \xi_{\gamma(1)}
$$
or equivalently
$$
\Psi_\gamma(T_{\gamma(0)} R_0) \pitchfork T_{\gamma(1)} Z_{R_1} \quad \text{\rm in} \, T_{\gamma(1)}M.
$$
\end{defn}
More generally, we consider the following situation.
We recall the definition of \emph{Reeb trace} $Z_R$ of a Legendrian submanifold
$$
Z_R: = \bigcup_{t \in \R} \phi_{R_\lambda}^t(R).
$$
(See \cite[Appendix B]{oh:contacton-transversality} for detailed discussion on its genericity.)
\begin{hypo}[Nondegeneracy]\label{hypo:nondegeneracy}
Let $\vec{R}=(R_1,\cdots,R_k)$ be a chain of Legendrian submanifolds.
We assume that we have
$$
Z_{R_i} \pitchfork R_j
$$
for all $i, \, j= 1,\ldots, k$.
\end{hypo}
Here $\phi^t_{R_\lambda}$ is the flow generated by the Reeb vector field $R_\lambda$.
We note that this nondegeneracy is equivalent the nondegeneracy of Reeb chords between $R_i$ and $R_j$.
(See \cite[Corollary 4.8]{oh:entanglement1}.)

\begin{cor}\label{cor:tangent-convergence} Let $w: \dot \Sigma \to M$ satisfy the contact instanton equation \eqref{eq:contacton-Legendrian-bdy-intro} and Hypothesis \eqref{eq:hypo-basic-pt}.
We assume that the pair $(\lambda, \vec R)$ is nondegenerate
in the sense of Hypothesis \ref{hypo:nondegeneracy}.
Then on each strip-like end with strip-like coordinates $(\tau,t) \in [0,\infty) \times [0,1]$ near a puncture
\beastar
&&\lim_{s\to \infty}\left|\pi \frac{\del w}{\del\tau}(s+\tau, t)\right|=0, \quad
\lim_{s\to \infty}\left|\pi \frac{\del w}{\del t}(s+\tau, t)\right|=0\\
&&\lim_{s\to \infty}\lambda(\frac{\del w}{\del\tau})(s+\tau, t)=0, \quad
\lim_{s\to \infty}\lambda(\frac{\del w}{\del t})(s+\tau, t)=T
\eeastar
and
$$
\lim_{s\to \infty}|\nabla^l dw(s+\tau, t)|=0 \quad \text{for any}\quad l\geq 1.
$$
All the limits are uniform for $(\tau, t)$ in $K\times [0,1]$ with compact $K\subset \R$.
\end{cor}
\begin{proof} The argument of the proof of this corollary is the same as that of
\cite[Corollary 5.11]{oh-wang:CR-map1}, which considers the closed string case.
Since we need to consider the boundary condition
and the vanishing of $Q$ is automatic in the present case,
we  provide the details of its proof here  for readers' convenience with minor adjustment
in details.

We first consider the first derivative estimate, i.e., the $C^1$-decay estimate.
If any of the above limits doesn't hold uniformly (take $|\pi \frac{\del w}{\del\tau}(s+\tau, t)|$ for example),
then there exists some $\epsilon_0>0$ and a sequence
$k\to \infty$, $(\tau_j, t_j)\in K\times S^1$ such that $|\pi \frac{\del w}{\del\tau}(s_k+\tau_j, t_j)|\geq \epsilon_0$.
Then we can take a subsequence limit $(\tau_j, t_j)\to (\tau_0, t_0)$ such that
$|\pi \frac{\del w}{\del\tau}(s_k+\tau_0, t_0)|\geq \frac{1}{2}\epsilon_0$ for $k$ large enough.

However, by Theorem \ref{thm:subsequence}, we can take a subsequence of $s_k$ such that
$w(s_k+\tau, t)$ converges to $\gamma(T\, t)$ in a neighborhood of $(\tau_0, t_0)\in K\times S^1$,
in the $C^\infty$ sense. Here $\gamma$ is some Reeb chord from $R$ to $R'$ for the
relevant pair $(R,R')$ associated to the given puncture.
Then we get $\lim_{s\to \infty}|\pi \frac{\del w}{\del\tau}(s_k+\tau_0, t_0)|=0$ and get a contradiction.

Once we establish this uniform $C^1$-decay result (which implies $W^{1,2}$-decay),
the higher order decay result is
an immediate consequence of the uniform local pointwise higher order
a priori estimates on the strip, of our Theorem \ref{thm:local-regularity}, since we can
cover the strip by a covering pair $\{U_\alpha\}$ and $\{V_\alpha\}$ with uniform size.
This finishes the proof.
\end{proof}

\section{Exponential $C^\infty$ convergence}
\label{sec:exponential-convergence}

In this section, we assume Hypothesis \ref{hypo:nondegeneracy}. Under the hypothesis,
we will improve the subsequence convergence to the exponential $C^\infty$ convergence
under the transversality hypothesis. Suppose that the tuple $\vec R= (R_0, \ldots, R_k)$ are
transversal in the sense all pairwise Reeb chords are nondegenerate. In particular we assume
that the tuples are pairwise disjoint.

We closely follow the arguments used in \cite[Section 11]{oh-wang:CR-map2}
which treats the closed string case \emph{under the hypothesis} of
vanishing charge $Q = 0$. Since in the current open string case, this vanishing is
proved and so all the arguments used therein can be repeated verbatim after adapting
them to the presence
of the boundary condition in the arguments. So we will outline the main
arguments of the proofs referring the detailed arguments to \cite[Section 11]{oh-wang:CR-map2}
but indicate only the necessary changes needed.

\subsection{$L^2$-exponential decay of the Reeb component of $dw$}
\label{subsec:Reeb}
We will prove the exponential decay of the Reeb component $w^*\lambda$.
We focus on a punctured neighborhood around a puncture $z_i \in \del \Sigma$ equipped with
strip-like coordinates $(\tau,t) \in [0,\infty) \times [0,1]$.

We again consider a complex-valued function $\alpha$ given in \eqref{eq:alpha}.
Then by the Legendrian boundary condition, we know $\alpha(\tau, i) \in \R$, i.e.
$$
\text{\rm Im}\, \alpha = 0
$$
for $i =0, \, 1$.

The following lemma was proved in the closed string case in
\cite{oh-wang:CR-map2} (in the more general context of Morse-Bott case).
For readers' convenience, we
provide some details by indicating how we adapt the argument
with the presence of boundary condition.

\begin{lem}[Compare with Lemma 11.20 \cite{oh-wang:CR-map2}]\label{lem:exp-decay-lemma}
Suppose the complex-valued functions $\alpha$ and $\nu$ defined on $[0, \infty)\times [0,1]$
satisfy
\beastar
\begin{cases}
\delbar \alpha = \nu, \\
\alpha(\tau, i) \in \R \,   \text{\rm for } i = 0,\, 1, \\
\|\nu\|_{L^2([0,1])}+\left\|\nabla\nu\right\|_{L^2([0,1])}\leq Ce^{-\delta \tau}
\, & \text{\rm for some constants } C, \delta>0\\
\lim_{\tau\rightarrow +\infty}\alpha(\tau,t) = T
\end{cases}
\eeastar
then $\|\alpha -T\|_{L^2(S^1)}\leq \overline{C}e^{-\delta \tau}$
for some constant $\overline{C}$.
\end{lem}
\begin{proof}
Notice that from previous section we have already
established the $W^{1, 2}$-exponential decay of $\nu = \frac{1}{2}|\zeta|^2$.
Once this is established, the proof of this $L^2$-exponential decay result
can be proved again by the standard three-interval method
and so omitted. See \cite[Theorem 10.11]{oh-wang:CR-map2} for a general abstract framework
which deals with the general Morse-Bott case. We also refer to
the Appendix of the arXiv version of \cite{oh-wang:CR-map1} for friendly details for the nondegenerate
case.
\end{proof}

\begin{rem} There is an error in the statement of \cite[Lemma 11.20]{oh-wang:CR-map2}
where
$$
\lim_{\tau\rightarrow +\infty}\alpha(\tau,t) = 0
$$
and the inequality $\|\alpha\|_{L^2(S^1)}\leq \overline{C}e^{-\delta \tau}$ appear:
This is an obvious error for which $\alpha$ should be replaced by $\alpha - T$ in
\cite[Lemma 11.20]{oh-wang:CR-map2}. (We thank a referee who pointed out this typo!)
Since the later proofs in \cite{oh-wang:CR-map2} are based on the correct statements,
we can apply them here too.
\end{rem}

\subsection{$C^0$ exponential convergence}

Now the $C^0$-exponential convergence of $w(\tau,\cdot)$ to some Reeb chord as $\tau \to \infty$
can be proved from the $L^2$-exponential estimates presented in previous sections by
the verbatim same argument as the proof of \cite[Proposition 11.21]{oh-wang:CR-map2}
with $S^1$ replaced by $[0,1]$ here. Therefore we omit its proof.

\begin{prop}[Compare with Proposition 11.21 \cite{oh-wang:CR-map2}]\label{prop:czero-convergence}
Under Hypothesis \ref{hypo:basic}, for any contact instanton $w$ satisfying the Legendrian boundary condition, there exists a unique Reeb orbit $\gamma$ such that the curve  $z(\cdot)=\gamma(T\cdot):[0,1] \to M$  satisfies
$$
\|d(w(\tau, \cdot), z(\cdot))
\|_{C^0([0,1])}\rightarrow 0,
$$
as $\tau\rightarrow +\infty$,
where $d$ denotes the distance on $M$ defined by the triad metric.
Here $T = T_\gamma $ is action of $\gamma$ given by  $T_\gamma =\int \gamma^*\lambda$.
\end{prop}

We just mention that the proof is based on the following lemma whose proof we
refer readers to that of \cite[Lemma 11.22]{oh-wang:CR-map2}.
\begin{lem} Let $t \in [0,1]$ be given. Then
for any given $\epsilon > 0$, there exists sufficiently large $\tau_1 > 0$ such that
$$
d(w(\tau,t), w(\tau', t)) < \epsilon
$$
for all $\tau, \, \tau' \geq \tau_1$.
\end{lem}

Then the following $C^0$-exponential convergence is also proved.
\begin{prop}[Compare with Proposition 11.23 \cite{oh-wang:CR-map2}]
There exist some constants $C>0$, $\delta>0$ and $\tau_0$ large such that for any $\tau>\tau_0$,
\beastar
\|d\left( w(\tau, \cdot), z(\cdot) \right) \|_{C^0([0,1])} &\leq& C\, e^{-\delta \tau}
\eeastar
\end{prop}

\subsection{$C^\infty$-exponential decay of $dw - R_\lambda(w) \, d\tau$}
\label{subsec:Cinftydecaydu}

 So far, we have established the following:
\begin{itemize}
\item $W^{1,2}$-exponential decay of $w$,
\item $C^0$-exponential convergence of $w(\tau,\cdot) \to z(\cdot)$ as $\tau \to \infty$
for some Reeb chord $z$ between two Legendrians $R, \, R'$.
\end{itemize}

Now we are ready to complete the proof of
$C^\infty$-exponential convergence $w(\tau,\cdot) \to z$
by establishing the $C^\infty$-exponential decay of $dw - R_\lambda(w)\, dt$.
The proof of the latter decay is now in order which will be
carried out by the bootstrapping arguments
applied to the system \eqref{eq:contacton-Legendrian-bdy-intro}.

Combining the above three, we have obtained $L^2$-exponential estimates of the full derivative $dw$.
By the bootstrapping argument using the local uniform a priori estimates
on the strip-like region as in the proof of Lemma \ref{lem:exp-decay-lemma},
we obtain higher order $C^{k,\alpha}$-exponential decays of the term
$$
\frac{\del w}{\del t} - T R_\lambda(z), \quad \frac{\del w}{\del\tau}
$$
for all $k\geq 0$, where $w(\tau,\cdot)$ converges to $z$ as $\tau \to \infty$ in $C^0$ sense.
Combining this, Lemma \ref{lem:exp-decay-lemma} and elliptic $C^{k,\alpha}$-estimates
given in Theorem \ref{thm:local-regularity}, we complete the proof of $C^\infty$-convergence of
$w(\tau,\cdot) \to z$ as $\tau \to \infty$.

\part{Index formula}
\label{part:index}

The first named author established a
relevant  Fredholm theory for the bordered contact instantons in \cite{oh:contacton-transversality}.
 We  follow the off-shell analytical framework of \cite{oh:contacton-transversality} which is in turn an
 adaptation thereof from  \cite{oh:contacton}, \cite{oh-savelyev}  to the current
 boundary value problem \eqref{eq:contacton-Legendrian-bdy-intro}.

 In this Part, we compute the Fredholm index of the linearization operator of
 \eqref{eq:contacton-Legendrian-bdy-intro} in the context of construction of Fukaya-type category
 on contact manifolds whose objects are Legendrian submanifolds. With this in mind, we
 consider disk contact instantons $w$. We will compute the index of its linearized operator
 in terms of some topological index, called the
 \emph{polygonal Maslov index} of the type given in \cite{fooo:anchored}. For this purpose,
 we need to explain the precise geometric framework with respect to which
 the gradings of the Floer-type complex defined for the whole system of
 moduli spaces become compatible under the gluing of the moduli spaces
 developed similarly as in \cite{oh:contacton-gluing}. For this purpose, we need to
 introduce the notion of \emph{graded anchored Legendrian submanifolds}.

\section{Off-shell framework of the moduli spaces}
\label{sec:off-shell}

Let $(\Sigma, j)$ be a bordered compact Riemann surface, and let $\dot \Sigma$ be the
punctured Riemann surface of disk-type with $\{z_1,\ldots, z_k \} \subset \del \Sigma$, we consider the moduli space
$$
\CM_k((\dot \Sigma,\del \dot \Sigma),(M,\vec R)), \quad \vec R = (R_1,\cdots, R_k)
$$
 of finite energy maps $w: \dot \Sigma \to M$ satisfying the equation
\eqref{eq:contacton-Legendrian-bdy-intro}.
\begin{rem} The correct notion of energy to be used for the study of moduli
spaces of contact instantons are given in \cite{oh:contacton}, \cite{oh:entanglement1}
which is used in the bubbling analysis of the contact instanton moduli spaces.
However the discussion of energy is largely unnecessary for our purpose of the computation of
index since the index calculation is made for a given contact instanton $w$ under the prescribed asymptotic
convergence to suitable Reeb chords under the nondegeneracy hypothesis as described below.
\end{rem}

According to the choice of strip-like coordinates,
we partition $\{z_1,\ldots, z_k \}$ into the union
$$
\{p^+_1,\ldots,p^+_{s^+},p^-_1,\ldots,p^-_{s^-}\}
$$
with $k = s^+ + s^-$ such that the (boundary) punctured surface
$$
\dot \Sigma = D^2 \setminus \{p^+_1,\ldots,p^+_{s^+},p^-_1,\ldots,p^-_{s^-}\} \subset \C
$$
carries the strip-like ends on pairwise disjoint open subsets thereof
\beastar
& &U_i^+ = \phi_i^+([0,\infty)\times[0,1])\subset \dot \Sigma \quad {\rm for} \; i=1,\ldots,s^+ \\
& &U_j^- = \phi_j^-((-\infty,0]\times[0,1])\subset \dot \Sigma \quad {\rm for} \; j=1,\ldots,s^-
\eeastar
where $\phi_i^+$'s and $\phi_j^-$'s are bi-holomorphic maps from the corresponding semi-strips into $\C$ respectively.

Under the conditions given in Hypothesis \ref{hypo:nondegeneracy},
we have the decomposition of the finite energy moduli space
$$
\CM_k((\dot \Sigma,\del \dot \Sigma),(M,\vec R)) =
\bigcup_{\vec \gamma \in \prod_{i=1}^{k}\frak{Reeb}(R_i,R_{i+1})} \CM(\vec \gamma),
\quad \gamma = (\gamma_1, \ldots, \gamma_k)
$$
where $i$ is counted modulo $k$. (See \cite{oh:contacton} for the closed
string case which can be easily adapted to the current bordered case as in
\cite{oh:entanglement1}.)

Let $\gamma^+_i$ for $i =1, \cdots, s^+$ and $\gamma^-_j$ for $j = 1, \cdots, s^-$
be two given collections of Reeb chords at positive and negative punctures
respectively. Following the notations from \cite{behwz},
we denote by $\underline \gamma$ and $\overline \gamma$ the corresponding
collections
\beastar
\underline \gamma & = & \{\gamma_1^+,\cdots, \gamma^+_{s^+}\} \\
\overline \gamma & = & \{\gamma_1^-,\cdots, \gamma^-_{s^-}\}.
\eeastar
For each $p_i$ (resp. $q_j$), we associate the strip-like
coordinates $(\tau,t) \in [0,\infty) \times S^1$ (resp. $(\tau,t) \in (-\infty,0] \times S^1$)
on the neighborhoods of boundary punctures.

We consider the set of smooth maps satisfying the boundary condition
\be\label{eq:bdy-condition}
w(z) \in R_i \quad \text{ for } \, z \in \overline{z_{i-1}z_i} \subset \del \dot \Sigma
\ee
and the asymptotic condition
\be\label{eq:asymp-condition}
\lim_{\tau \to \infty}w((\tau,t)_i) = \gamma^+_i(T_i(t+t_i)), \qquad
\lim_{\tau \to - \infty}w((\tau,t)_j) = \gamma^-_j(T_j(t-t_j))
\ee
for some $t_i, \, t_j \in S^1$, where
$$
T_i = \int_{S^1} (\gamma^+_i)^*\lambda, \, T_j = \int_{S^1} ( \gamma^-_j)^*\lambda.
$$
Here $t_i,\, t_j$ depends on the given analytic coordinate and the parameterizations of
the Reeb chords.

We will fix a domain complex structure $j$ and its associated K\"ahler metric $h$.
We regard the assignment
$$
\Upsilon: w \mapsto \left(\delbar^\pi w, d(w^*\lambda \circ j)\right), \quad
\Upsilon: = (\Upsilon_1,\Upsilon_2)
$$
as a section of the (infinite dimensional) vector bundle: Using the decomposition
$TM = \xi \oplus \R\langle R_\lambda \rangle$, we write $\Upsilon = (\Upsilon_1, \Upsilon_2)$.
Then we can regard the assignment
$$
\Upsilon_1: w \mapsto \delbar^\pi w
$$
as a smooth section of some infinite dimensional vector bundle and
the assignment
$$
\Upsilon_2: w \mapsto d(w^*\lambda \circ j)
$$
as a map to $\Omega^2(\Sigma)$.
(We refer to \cite{oh:contacton-transversality} for the precise off-shell framework
for these operators.)

We consider the linearized operator
$$
D\Upsilon(w): \Omega^0(w^*TM,(\del w)^*T\vec R) \to \Omega^{(0,1)}(w^*\xi) \oplus \Omega^2(\Sigma)
$$
of each contact instanton $w$, i.e., $\Upsilon(w) = 0$.

Then the following explicit formula thereof is derived in \cite{oh:contacton} for the closed
string case which also holds with Legendrian boundary condition added.
(We alert readers that the notation $B$ used here is nothing to do with the $B$ used before
to denote the second fundamental form.)

\begin{thm}[Theorem 10.1 \cite{oh:contacton}; See also Theorem 1.15
\cite{oh-savelyev}] \label{thm:linearization} In terms of the decomposition
$d\pi = d^\pi w + w^*\lambda\, R_\lambda$
and $Y = Y^\pi + \lambda(Y) R_\lambda$, we have
\bea
D\Upsilon_1(w)(Y) & = & \delbar^{\nabla^\pi}Y^\pi + B^{(0,1)}(Y^\pi) +  T^{\pi,(0,1)}_{dw}(Y^\pi)\nonumber\\
&{}& \quad + \frac{1}{2}\lambda(Y) (\CL_{R_\lambda}J)J(\del^\pi w)
\label{eq:DUpsilon1}\\
D\Upsilon_2(w)(Y) & = &  - \Delta (\lambda(Y))\, dA + d((Y^\pi \rfloor d\lambda) \circ j)
\label{eq:DUpsilon2}
\eea
where $B^{(0,1)}$ and $T_{dw}^{\pi,(0,1)}$ are the $(0,1)$-components of $B$ and
$T_{dw}^\pi$, where $B, \, T_{dw}^\pi: \Omega^0(w^*TM) \to \Omega^1(w^*\xi)$ are
 zero-order differential operators given by
\be\label{eq:B}
B(Y) =
- \frac{1}{2}  w^*\lambda \otimes \left((\CL_{R_\lambda}J)J Y\right)
\ee
and
\be\label{eq:torsion-dw}
T_{dw}^\pi(Y) = \pi T(Y,dw)
\ee
respectively.
\end{thm}

We  recall the Propostion \ref{prop:open-fredholm}
stated in Section \ref{subsec:index-formula} of the present paper
which states the Fredholm property of the linearized operator
for the bordered contact instantons with Legendrian boundary condition.
This will be the basis of our index computation in the next few sections.

Before launching on our computation,  we first
remind readers of the decomposition $Y = Y^\pi + \lambda(Y)\, R_\lambda$.
Noting that $Y^\pi$ and $\lambda(Y)$ are independent of each other, we will write
\be\label{eq:Y-decompose}
Y = Y^\pi + f R_\lambda, \quad f: = \lambda(Y)
\ee
where $f: \dot \Sigma \to \R$ is an arbitrary function satisfying the boundary condition
\be\label{eq:Y-bdy-condition}
Y^\pi(\del \dot \Sigma) \subset T\vec R, \quad f|_{\del \dot \Sigma} = 0
\ee
by the Legendrian boundary condition satisfied by $Y$.

This wraps up the discussion of the Fredholm property of
the linearization map
$$
D\Upsilon_{(\lambda,T)}(w): \Omega^0_{k,p}(w^*TM, (\del w)^*T\vec R;J;\underline \gamma,\overline \gamma) \to
\Omega^{(0,1)}_{k-1,p}(w^*\xi) \oplus \Omega^2_{k-2,p}(\dot \Sigma)
$$
which is borrowed from \cite{oh:contacton-transversality}.

In the following three sections, we shall compute the (Fredholm) index of the linearized operator $D\Upsilon(w)$ in terms of a topological index, called the polygonal Maslov
index. However in order to do this, we need more data. So we shall use the notion of
anchors as in \cite{fooo:anchored}.

\section{Anchored Legendrian links}
\label{sec:anchored}

\begin{defn}
Fix a base point $y$ of ambient contact manifold $(M,\xi)$. Let $R$ be a Legendrian submanifold of $(M,\xi)$. We define an
\emph{anchor} of $R$ to $y$ is a path $\ell:[0,1] \rightarrow M$ such that $\ell(0)=y,\;\ell(1)\in R$. We call a pair $(R,\ell)$ an \emph{anchored} Legendrian submanifold.
A chain $\mathcal{E} = ((R_1,\ell_1),\ldots,(R_k,\ell_k))$ is called an \emph{anchored Legendrian chain}.
\end{defn}

For given $\mathcal{E} = ((R_1,\ell_1),\ldots,(R_k,\ell_k))$, we define the set of chords
$$
\Omega(R_i,R_j) = \{\ell:[0,1] \rightarrow M \mid \ell(0)\in R_i,\;\ell(1)\in R_j\}
$$
for each pair $i, \, j$ with $i = j$ allowed.
Then these anchors give a systematic choice of a base path $\ell_{ij}\in\Omega(R_i,R_j)$ by concatenating $\ell_i$ and $\ell_j$ i.e.,
\be\label{eq:base-path}
\ell_{ij}(t) = \overline{\ell}_i\ast\ell_j(t) =
\left\{
\begin{array}{ll}
\ell_i(1-2t) & t\leq 1/2 \\
\ell_j(2t-1) & t\geq 1/2
\end{array} \right.
\ee
The upshot of this construction is the following overlapping property
\be\label{eq:overlapping}
\ell_{ij}(t)=\ell_{il}(t) \quad {\rm for} \; 0\leq t\leq \frac{1}{2},
\qquad \ell_{ij}(t) = \ell_{lj}(t) \quad {\rm for} \; \frac{1}{2} \leq t \leq 1
\ee
for all $i,j,l$.

Let $\vec{R} = (R_1,\ldots,R_k)$ be a Legendrian chain. We consider the set of homotopy class of maps $w:\dot \Sigma \rightarrow M$ satisfying \eqref{eq:bdy-condition}-\eqref{eq:asymp-condition}. We denote it by $\pi_2(\vec{R};\underline \gamma,\overline \gamma)$. If $\mathcal{E} = (\vec{R},\vec{\ell})$ is an anchored Legendrian chain then we write $\pi_2(\mathcal{E};\underline \gamma,\overline \gamma)$
in place of $\pi_2(\vec{R};\underline \gamma,\overline \gamma)$ by abuse of notation.

\begin{defn}\label{defn:admissible}
Let $\mathcal{E} = \{(R_i,\ell_i)\}_{1\leq i\leq k}$ be a chain of anchored Legendrian submanifolds. A homotopy class $B \in \pi_2(\mathcal{E};\underline \gamma,\overline \gamma)$ is called
\emph{admissible} to $\mathcal{E}$ if it can be obtained by a polygon that is a gluing of $k$ bounding strips $w_{(i-1)i}^+:[0,1]\times[0,1] \rightarrow M$ satisfying
\bea\label{eq:pos-admissible-piece}
& &w_{(i-1)i}^+(0,t)  = \gamma_i(t) \nonumber\\
& &w_{(i-1)i}^+(1,t) =  \ell_{(i-1)i}(t) \quad (= \overline{\ell}_{i-1}\ast\ell_i(t) )\\
& &w_{(i-1)i}^+(\tau,0)  \in  R_{i-1}, \quad w_{(i-1)i}^+(\tau,1)\in R_i\nonumber
\eea
\end{defn}

When this is the case, we denote the homotopy class $B$ as
$$
B = [w_{12}^+]\#\ldots\#[w_{k1}^+]
$$
and the set of admissible homotopy classes by $\pi^{ad}_2(\mathcal{E};\underline \gamma,\overline \gamma)$.
In order to define a polygonal index, we need an additional data as in \cite{fooo:anchored}
the explanation of which is now in order.

Let $y\in M$ be the base point given above.  We fix a Lagrangian subspace $V_y\in Lag(\xi_y)$.

\begin{defn}
Consider an anchored Legendrian $(R,\gamma)$. We denote by $\alpha$ a section of
$\gamma^*Lag(\xi,d\lambda)$ such that
$$
 \alpha(0)=V_y, \quad \alpha(1)=T_{\gamma(1)}R \subset \xi_{\gamma(1)}
 $$
We call such a pair $(\gamma,\alpha)$ a \emph{graded anchor} of $R$ (relative to $(y,V_y)$) and a triple $(R,\gamma,\alpha)$ a \emph{graded anchored Legendrian submanifold}.
\end{defn}

To proceed further, we recall the construction of Maslov index (see \cite{arnold:index} or \cite{fooo:anchored})
and introduce the notion of grading of anchors (see \cite{fooo:anchored}).

Let $(S,\omega)$ be a symplectic vector space. Then we denote by
$$Lag(S,\omega) = \{V \mid V \; {\rm is} \; {\rm a} \; {\rm Lagrangian} \; {\rm subspace} \; {\rm of} \; (S,\omega)\}$$
the Lagrangian Grassmannian of $(S,\omega)$ and let
$$
Lag_1(S,\omega) = \{V\in Lag(S,\omega) \mid \dim(V\cap V_0)\geq 1\}
$$
be the Maslov cycle whose
 whose Poincar\'e dual represents the Maslov class $\mu\in H^1(Lag(S,\omega),\Z)$ \cite{arnold:index}.

The tangent space $T_{V_0}Lag(S,\omega)$ is canonically isomorphic to the set of quadratic forms on $V_0$.

\begin{defn}\label{defn:positively-directed}
We say any tangent vector corresponding a nondegenerate positive-definite quadratic form is \emph{positively directed}.
\end{defn}

The following is a useful proposition for construction of the polygonal Maslov index. We refer to \cite{arnold:index} for its proof.

\begin{prop}\label{prop:directed-lag-path}
Let $(S,\omega)$ be a symplectic vector space and $V_1 \in (S,\omega)$ be a given Lagrangian subspace. Let $V_0 \in Lag(S,\omega)\setminus Lag_1(S,\omega;V_1)$ i.e., be a Lagrangian subspace with $V_0\cap V_1=\{0\}$. Consider smooth paths $\alpha:[0,1] \rightarrow Lag(S,\omega)$ satisfying
\begin{enumerate}
\item $\alpha(0)=V_0,\; \alpha(1)=V_1$.
\item $\alpha(t)\in Lag(S,\omega)\setminus Lag_1(S,\omega;V_1)$ for all $0\leq t<1$.
\item $-\alpha'(1)$ is positively directed.
\end{enumerate}
Then any two paths $\alpha_1,\alpha_2$ are homotopic to each other via a homotopy $\alpha_s$ satisfying the 3 conditions above for each $s$.
\end{prop}

\begin{rem}
Our polygon will have the boundary decorated by Legendrian submanifolds $R_1,\ldots,R_k$ with counter-clockwise order but in \cite{fooo:anchored} Lagrangian submanifolds $L_1,\ldots,L_k$ has the clockwise order. So for the consistency with that of
\cite{fooo:anchored}, we choose the opposite direction of $\alpha$ in the above proposition to that of path in \cite[Proposition 5.3]{fooo:anchored}. More explicitly, if we take $\overline{\alpha}(t) = \alpha(1-t)$ instead of $\alpha(t)$ then our proposition is exactly same as \cite[Proposition 5.3]{fooo:anchored}.
\end{rem}

Let $(R_0,\ell_0,\alpha_0)$ and $(R_1,\ell_1,\alpha_1)$ be a pair of
graded anchored Legendrian submanifolds relative to $(y,V_y)$ that satisfies
transversality stated in Hypothesis \ref{hypo:nondegeneracy} which is applied to the pair, i.e.,
$Z_{R_0} \pitchfork R_1$.

Then we can define the path $\alpha_{01}(t) \in Lag(\xi_{\ell_{01}(t)})$
by concatenating $\alpha_0$ and $\alpha_1$ i.e.,
$$
\alpha_{01} = \overline{\alpha}_0\ast\alpha_1.
$$

Now consider a Reeb chord $\gamma$ from $R_0$ to $R_1$ and a map $w:[0,1]^2 \rightarrow M$ satisfying the boundary condition
\bea\label{eq:admissible-piece}
& &w(\tau,0) \in  R_0, \; w(\tau,1)\in R_1, \label{eq:bdy-R0R1}\\
& &w(0,t) =  \ell_{01}(t), \; w(1,t)=\gamma(t). \label{eq:asymp-gamma}
\eea
Denote by $[w,\gamma]$ the homotopy class of such $w$'s.
Similarly we denote by $[\gamma,w]$ the homotopy class of such $w$'s satisfying
\eqref{eq:bdy-R0R1} and \eqref{eq:asymp-gamma} replaced by
$$
w(0,t)=\gamma(t), \quad w(1,t)=\ell_{01}(t)
$$
instead.

\begin{rem}\label{rem:[w,gamma]} By the exponential convergence from Section \ref{sec:exponential-convergence},
we can compactify the map contact instanton $w: [0,\infty) \times [0,1] \to M$
satisfying
$$
w(0,t)=\ell_{01}(t), \; w(\tau,0)\in R_0, \; w(\tau,1)\in R_1 \; \lim_{\tau\rightarrow\infty}w(\tau,t) = \gamma(t)
$$
to a continuous map from $\overline w: [0,\infty) \cup \{\infty\} \times [0,1] \cong [0,1]^2$. This latter map
defines a  relative homotopy class $[w,\gamma]$ as that of a map $\overline w$
where the condition at $t =1$ for \eqref{eq:admissible-piece} is replaced by the
asymptotic condition at $\tau =\infty$. We denote by
$$
[w,\gamma]
$$
the homotopy class of this map $\overline w$. The analytical index of the contact instanton $w$ will
be expressed in terms of some topological index of this compactified map $\overline w$.
(See \cite{fooo:anchored} for a similar practice.)
\end{rem}

Now we are ready to assign an integer grading to each Reeb chord $\gamma \in \frak{Reeb}(R_0,R_1)$
when $R_i$'are graded anchored.

Recalling the decomposition \eqref{eq:Y-decompose} of $Y$ arising from $TM = \xi \oplus \R \langle R_\lambda \rangle$,
we can choose a trivialization $\Phi:w^*TM \rightarrow [0,1]^2\times\R^{2n+1}$
by taking the diagonal form of the trivialization as follows.

\begin{defn}[Diagonal trivialization $\Phi$]\label{defn:Phi} $\Phi$ satisfies
\begin{itemize}
\item it respects the splitting $w^*TM = w^*\xi \oplus \R\langle R_\lambda \rangle$ and
$$
\R^{2n+1} \cong \C^n \oplus \R,
$$
\item  $\Phi|_{w^*\xi}:w^*\xi \rightarrow [0,1]\times\R^{2n}\times\{0\} \cong [0,1]^2\times\R^{2n}$ is a symplectic trivialization.
\end{itemize}
\end{defn}

Then we can associate a Lagrangian path $\alpha^\Phi_{[w,\gamma];\alpha_{01}}$ defined on $\partial[0,1]^2\setminus \{(1,1)\}$ by
\be\label{eq:lag-loop-piece}
\begin{cases}
\alpha^\Phi_{[w,\gamma];\alpha_{01}}(\tau,i)  =  \Phi(T_{w(\tau,i)}R_i) \qquad {\rm for}
\; i=0,1 \nonumber \\
\alpha^\Phi_{[w,\gamma];\alpha_{01}}(0,t)  =  \Phi(\alpha_{01}(t)), \\
 \alpha^\Phi_{[w,\gamma];\alpha_{01}}(1,t)=\Phi(d\phi^{T_0t}_{R_\lambda}(T_{\gamma(0)}R_0)).
\end{cases}
\ee
Here $T_0$ is defined by $\gamma(t)=\phi^{T_0t}_{R_\lambda}(\gamma(0))$. Note that it is not necessarily nonnegative. If $T_0 < 0$, then $\gamma(t)$ is opposite with respect to the Reeb vector field. Note that
$$
\lim_{t\rightarrow 1}\alpha^\Phi_{[w,\gamma];\alpha_{01}}(1,t)=\Phi(d\phi^{T_0}_{R_\lambda}(T_{\gamma(0)}R_0))$$
and
$$
\lim_{\tau\rightarrow 1}\alpha^\Phi_{[w,\gamma];\alpha_{01}}(\tau,1)=\Phi(T_{w(\tau,1)}R_1)
$$
 are transversal to each other by the nondegeneracy hypothesis given in
 Hypothesis \ref{hypo:nondegeneracy}.
Then we have a Lagrangian path $\alpha^-$ from $\Phi(d\phi^{T_0}_{R_\lambda}(T_{\gamma(0)}R_0))$ to $\Phi(T_{w(\tau,1)}R_1)$ satisfying the conditions in Proposition \ref{prop:directed-lag-path}. Concatenating these two paths, we have a Lagrangian loop $\widetilde{\alpha}^\Phi_{[w,\gamma];\alpha_{01}}$ in $\R^{2n}$.

\begin{defn}[Grading of chords]
We define $\mu([w,\gamma];\alpha_{01})$ to be the Maslov index of this Lagrangian loop $\widetilde{\alpha}^\Phi_{[w,\gamma];\alpha_{01}}$ in $(\R^{2n},\omega_0)$.
\end{defn}
Note that this index is independent of the choice of trivializations $\Phi$
and depends only on the admissible homotopy classes.
In general, this defines a $\Z/\Gamma$ grading where $\Gamma = \Gamma_{R_0R_1}$ is the positive generator of the
abelian group
\be\label{eq:Gamma}
\Gamma_{R_0R_1}: = \{ \mu(A) \in \Z \mid A \in \pi_1(\Omega(R_0,R_1))\}
\ee
where $\mu(A)$ is the Maslov index of the annulus map $u: S^1 \times [0,1] \to M$
satisfying the boundary condition
$$
u(\theta,0) \in R_0, \, u(\theta,1) \in R_1 \quad \theta \in S^1.
$$

\section{Polygonal Maslov index}
\label{sec:polygonal}

Let $\dot \Sigma = D^2 \setminus \{z_1,\ldots,z_k\}$, $\vec{R} = (R_1,\ldots,R_k)$, $(\underline \gamma,\overline \gamma)$ and $\mathcal{F}((\dot{\Sigma},\del \dot{\Sigma}), \mathcal{E}; J; \underline \gamma,\overline \gamma)$ as above.

In this section, we define a polygonal Maslov index $\mu(\mathcal{E},\underline \gamma,\overline \gamma;B)$ associated to each homotopy class $B\in\pi_2(\mathcal{E};\underline \gamma,\overline \gamma)$. (See Remark \ref{rem:[w,gamma]}.)
We consider maps $w\in \mathcal{F}((\dot{\Sigma},\del \dot{\Sigma}), \mathcal{E}; J; \underline \gamma,\overline \gamma)$ with $[w]=B\in\pi_2(\mathcal{E};\underline \gamma,\overline \gamma)$. We again choose a diagonal trivialization $\Phi:w^*TM \rightarrow \dot \Sigma\times\R^{2n+1}$ as in Definition \ref{defn:Phi}.
Under the symplectic trivialization $\Phi|_{w^*\xi}$, we have $k$ smooth Lagrangian paths
$$
\alpha_i:\overline{z_{i-1}z_i}\subset \partial\dot \Sigma \rightarrow Lag(\R^{2n},\omega_0);
\qquad \alpha_i(t) = \Phi(T_{w(t)}R_i) =: \Lambda_i(t)
$$
for $i=1,\ldots,k$ and $k$ smooth Lagrangian paths
$$ \beta_i:[0,1] \rightarrow Lag(\R^{2n},\omega_0); \qquad \beta_i(t) = \Phi(T\psi^{T_it}_{R_\lambda}(T_{\gamma_i(0)}R_i))$$
Here $z_i = p_{i+1}^+$ for $i=1,\ldots,s^+$ and $z_j = p_{j-s^+}^-$ for $j=s^++1,\ldots,k$ which are punctures of $D^2$, and $\overline{z_{i-1}z_i}$ is an open arc from $z_{i-1}$ to $z_i$ and $T_i$ is defined by $\gamma_i(t) = \psi^{T_it}_{R_\lambda}(\gamma_i(0))$.

In order to construct a Lagrangian loop we need to construct more $k$ Lagrangian paths. By Hypothesis \ref{hypo:nondegeneracy},
$$ \beta_{i-1}(1) \pitchfork \alpha_i(z_{i-1}):= \lim_{t\rightarrow z_{i-1}+0}\alpha_i(t)$$
in $\R^{2n}$. Now we have $\overline{\alpha}:[0,1] \rightarrow Lag(\R^{2n},\omega_0)$ satisfying the conditions in Proposition \ref{prop:directed-lag-path} for $i=1,\ldots,k$ so that $\overline{\alpha}(0)=\beta_{i-1}(1)$ and $\overline{\alpha}(1)=\alpha_i(z_{i-1})$.

Now we have a Lagrangian path
$$\widetilde{\alpha}_w = \overline{\alpha}_1\ast\alpha_1\ast\beta_1\ast\cdots\ast
\overline{\alpha}_k\ast\alpha_k\ast\beta_k$$
in $Lag(\R^{2n},\omega_0)$.

\begin{defn}
We define the polygonal index $\mu(\mathcal{E},\underline \gamma,\overline \gamma;B)$ by the Maslov index of the Lagrangian loop $\widetilde{\alpha}_w$ i.e.,
$$\mu(\mathcal{E},\underline \gamma,\overline \gamma;B) = \mu(\widetilde{\alpha}_w).$$
\end{defn}

Note that the definition is independent of the choice of the trivialization $\Phi$ and that we did not use the anchors yet in the construction of the index. Assume that $\mathcal{E} = \{(R_i,\ell_i,\alpha_i)\}_{1\leq i\leq k}$ is a chain of graded anchored Legendrian submanifolds. It induces a grading $\alpha_{ij} = \overline{\alpha}_i\ast\alpha_j$ along $\ell_{ij} = \overline{\ell}_i\ast\ell_j$ for $i,j$. Then $\alpha_{ij}$ also satisfy the overlapping property
\be\label{eq:graded-overlapping}
\alpha_{ij}|_{[0,1/2]}=\alpha_{il}|_{[0,1/2]} \qquad \alpha_{ij}|_{[1/2,1]}=\alpha_{lj}|_{[1/2,1]}
\ee
for all $i,j$.

Recall that $B\in \pi^{ad}_2(\mathcal{E},\underline \gamma,\overline \gamma)$ i.e.,
$$
B = [w_{12}^+]\#\ldots\#[w_{k1}^+]
$$
where $w_{(i-1)i}^+$ are defined in Definition \ref{defn:admissible}.
By definition of $[w_{(i-1)i}^+,\gamma_i]$ and $\alpha_{ij}$, we can define the Maslov indices $\mu([w_{(i-1)i}^+,\gamma_i];\alpha_{(i-1)i})$.

Now we have the following proposition whose proof is similar to \cite[Lemma 5.12]{fooo:anchored}
and so omitted.

\begin{prop}\label{prop:decomp-index}
Let $\mathcal{E}$ be a graded anchored Legendrian chain. Suppose $B\in\pi^{ad}_2(\mathcal{E},\underline \gamma,\overline \gamma)$. Then we have
$$ \mu(\mathcal{E},\underline \gamma,\overline \gamma;B) = \sum_{i=1}^{k}\mu([w_{(i-1)i}^+,\gamma_i];\alpha_{(i-1)i}). $$
\end{prop}

\section{Calculation of Fredholm index }
\label{sec:index}

Let $\mathcal{E} = \{(R_i,\ell_i,\alpha_i)\}_{1 \leq i \leq k}$ be a graded anchored Legendrian chain. Recall the operator
$$\Upsilon: \CW^{k,p} \to \CH^{(0,1)}_{k-1,p} (\xi;\vec R) $$
defined in Section \ref{sec:off-shell}. From now on we restrict $\Upsilon$ to
$$
\CW^{k,p} := \CW^{k,p}_{ad}((\dot{\Sigma},\del \dot{\Sigma}), \mathcal{E}; \underline \gamma, \overline \gamma)
$$
which consists of elements $w$ with $[w]=B$ admissible.

Recall that the linearization $D\Upsilon(w)$ of $\Upsilon$ is homotopic to the diagonal operator in \eqref{eq:diagonal}
with Legendrian boundary condition.
Denote this operator by $L$. Note that this homotopy preserves the index. Therefore we have
$$
{\rm Index }D\Upsilon(w) = {\rm Index } L = {\rm Index }(\bar{\partial}^{\nabla^\pi}+B^{(0,1)}
+T^{w,(0,1)}_{dw}) + {\rm Index }(-\Delta)
$$
Hence we suffices to compute the index of the operator $L$. To distinguish notations, we denote it by $L_w$ In order to compute this, we shall use the gluing formula. More explicitly, consider the operator $L$ on the half-infinite strip
$$
Z^+ := (-\infty,0]\times[0,1] \quad \text{for $ i=1,\ldots,s^+$}
$$
(resp. $Z^- := [0,\infty)\times[0,1]$ for $i=s^++1,\ldots,k$) with boundary conditions:

\bea\label{eq:half-strip-leg-bdry}
Y(0,t)\in \alpha_{(i-1)i}(t), \; Y(\tau,0)\in TR_{i-1}, \; Y(\tau,1)\in TR_i
\eea

We denote by  $L_{w_{(i-1)i}^+}$ on $Z^\pm$. Then we can glue all operators and we obtain an operator
$$
L_{\rm glued} := L_w \# \sum_{i=1}^{k}L_{w_{(i-1)i}^+}
$$
satisfying (linear) Legendrian boundary condition,
whose domain is isomorphic to a closed disk.
By the gluing formula of the Fredholm index, we obtain
$$
{\rm Index } L_w = {\rm Index }L_{\rm glued} - \sum_{i=1}^{k}{\rm Index }L_{w_{(i-1)i}^+}.
$$
We now compute the indices of each summand of the right hand side.
Before proceeding further, we remark the following
standard fact:
Although we have not explicitly mentioned, we need to take a certain $W^{k,p}$
completion with $k \geq 2$ for the study of index of the linearized operator.
Therefore we need to put the condition
\be\label{eq:finiteW2pnorm}
\|Y\|_{W^{2,p}} < \infty
\ee
(See \cite{oh:contacton-transversality} for the precise off-shell framework of
the Fredhlom theory of the linearized operator.)

\subsection{Computation for the glued operator}

We first compute the index of the glued operator. Note that the topological type of
the glued domain is a closed disk, which we denote by $D^2$ below.

Using the coordinate invariance of the principal symbol of the differential operator
that by construction and using the decomposition \eqref{eq:Y-decompose}, we know that
the operator can be homotoped to the direct sum of
\begin{itemize}
\item a second order differential operator whose symbol is
the same as the Laplacian $-\Delta$ on $D^2$ with Dirichlet boundary condition acted upon
$f = \lambda(Y)$. We denote it by $D_2$.
\item a linearized Cauchy-Riemann type operator  on $D^2$ with Lagrangian boundary condition
acted upon $Y^\pi$ which is determined by a certain Lagrangian loop.
In fact it is easy to see that this loop is homotopic
to a constant loop by admissibility hypothesis of $[w]$.
We denote by the resulting operator by $D_1$.
\end{itemize}
(We refer readers \cite[Appendix]{oh:cag} and \cite{fooo:anchored} for similar practice
for the Cauchy-Riemann. type operators on $Z^-$ in symplectic geometry.)

Therefore we obtain
$$
{\rm Index }L_{\rm glued} = {\rm Index }(L_2) + {\rm Index }(L_1) = 0 + n = n:
$$
Here we use the standard facts that
that principal symbol determines the ellipticity of linear differential operators
and that the Fredholm index is preserved under a continuous family of
Fredholm operators. For the computation $ {\rm Index } L_2$, we note that $L_2$
can be deformed to the Dirichelet Laplacian $-\Delta_0$ as an elliptic family
on the disc. Therefore it has index 0. For the operator $L_1$ on $D^2$,
we can again deform it to  the Cauchy-Riemann operator $\delbar$
(by deforming away its zero order part ) for the contractible Lagrangian loop as its boundary
condition. This implies that the operator has index $n$ by the standard
index calculation for the Riemann-Hilbert operator. (See \cite{oh:kmj}, for example.)
(This computation is just in the realm of the index theory of elliptic boundary
value problem compact manifold with boundary from \cite{atiyah-bott}. In our case the domain
is just $D^2$.)

\subsection{Computation of index on the semi-strips}

Now it remains to compute the index of ${\rm Index }L_{w_{(i-1)i}^+}$'s each of which
has the form $L_{[w,\gamma]}$ defined on $Z^-$ with Legendrian boundary conditions
$$
 Y(0,t)\in \alpha_{01}(t), \; Y(\tau,0)\in TR_0, \; Y(\tau,1)\in TR_1
$$
and \eqref{eq:finiteW2pnorm}. By the homotopy invariance over the homotopy
\eqref{eq:s-homotopy}, we have
\beastar
{\rm Index }L_{[w,\gamma]} & = & {\rm Index } \left(\begin{matrix}\bar{\partial}^{\nabla^\pi}+B^{(0,1)}+T^{w,(0,1)}_{dw} & 0 \\
0 & -\Delta_0 \end{matrix}\right)\\
& = & {\rm Index }(\bar{\partial}^{\nabla^\pi}+B^{(0,1)}+T^{w,(0,1)}_{dw}) + {\rm Index }(-\Delta_0).
\eeastar
(On $Z^+$, we can do the change of variables $(\tau,t) \mapsto (-\tau,t)$
to the operator on $Z^-$ and so it is enough to consider the case of $Z^-$.)

\subsubsection{Contribution of the Reeb component}

We compute ${\rm Index }(-\Delta_0)$ first. By the (linear) elliptic regularity, it will be enough
take the $W^{2,p}$-completion with $p > 2$ as its off-shell function for the operator $\Delta_0$,
whose kernel will be solutions of the Dirichlet problem of the form:
\be\label{eq:ker(-Delta0)}
\begin{cases}
\Delta f=0, \quad \|f\|_{W^{2,p}} < \infty,\\
f(\tau,0)=f(\tau,1)=f(0,t)=0,
\end{cases}
\ee
where $f:Z^-\subset \C \rightarrow \R$.
\begin{lem} $\Index (-\Delta_0) = 0$.
\end{lem}
\begin{proof} For the proof,  we first compute the formal-adjoint of $-\Delta_0$.
This formal adjoint equation can  is defined as follows: for $g \in L^q$ with $1/p + 1/q = 1$
$$
\int_{Z^-} (-\Delta f\cdot g)\, dA = 0
$$
for all $f \in C^\infty_0(Z^-)$ satisfying the same boundary condition as that of
\eqref{eq:ker(-Delta0)}.

By (formally) doing integration by parts, we compute
\beastar
\int_{Z^-} (-\Delta f\cdot g)dA &=& \int_{Z^-} \big(-\frac{\partial^2 f}{\partial \tau^2}-\frac{\partial^2 f}{\partial t^2}\big) \cdot g dA \\
& = & -\int_{Z^-} \frac{\partial^2 f}{\partial \tau^2}\cdot gdA - \int_{Z^-} \frac{\partial^2 f}{\partial t^2}\cdot gdA \\
&=& -\int_{Z^-} f\cdot\frac{\partial^2 g}{\partial \tau^2}dA - \int_{Z^-} f\cdot\frac{\partial^2 g}{\partial t^2}dA
+\int_0^1 \frac{\partial f}{\partial \tau}(0,t)\cdot g(0,t)dt \\
& &+ \int_0^\infty \frac{\partial f}{\partial t}(\tau,0)\cdot g(\tau,0)d\tau - \int_0^\infty \frac{\partial f}{\partial t}(\tau,1)\cdot g(\tau,1)d\tau \\
&=& \int_{Z^-} f\cdot(-\Delta g)dA +\int_0^1 \frac{\partial f}{\partial \tau}(0,t)\cdot g(0,t)dt \\
& & + \int_0^\infty \frac{\partial f}{\partial t}(\tau,0)\cdot g(\tau,0)d\tau - \int_0^\infty \frac{\partial f}{\partial t}(\tau,1)\cdot g(\tau,1)d\tau \\
&=&0.
\eeastar
Since $f$ is arbitrary, we have obtained equation
\be\label{eq:coker(-Delta0)}
\begin{cases}
-\Delta g=0, \quad \|g\|_{L^q} < \infty \\
g(\tau,0)=g(\tau,1)=g(0,t)=0.
\end{cases}
\ee
By the elliptic regularity, the solutions of \eqref{eq:ker(-Delta0)} and \eqref{eq:coker(-Delta0)}
are smooth and the space of solutions are the same.
This implies $\dim {\rm ker}(-\Delta_0)^\dagger = \dim {\rm ker}(-\Delta_0)$.
Since ${\rm coker}(-\Delta_0) \cong \ker(-\Delta_0)^\dagger$,
we conclude 
$$
\Index (-\Delta_0) = \dim {\rm ker}(-\Delta_0) -  \dim {\rm coker}(-\Delta_0) = 0
$$
which finishes the proof of the lemma.
\end{proof}

\subsubsection{Contribution of the $\xi$ component}

Now we shall compute the index of the operator
$$
\bar{\partial}^{\nabla^\pi}+B^{(0,1)}+T^{\pi,(0,1)}_{dw}
$$
by following  \cite[Appendix]{oh:cag}. Note that it is in fact the restriction of the linearization $D\Upsilon_1$ to $\Omega^0(w^*\xi;TR_0,TR_1;\alpha_{01})$.

To compute the linearization $D\Upsilon_1$ we again use the trivialization $\Phi$ which
we used in the construction of polygonal Maslov index.
Then we have the push forward operator
$$\Phi_\ast D\Upsilon_1:W^{1,2}(Z^-,\R^{2n+1};\Lambda_0,\Lambda_1;\alpha^\Phi_{01}) \rightarrow L^2(Z^-,\R^{2n})$$
where $\Lambda_i(\tau) = \Phi(T_{w(\tau,i)}R_i)$ for $i=0,1$ and $\alpha^\Phi_{01} = \Phi(\alpha_{01})$.
Then this operator gives an operator
$$
\Phi_\ast(D\Upsilon_1|_{w^*\xi}):W^{1,2}(Z^-,\R^{2n};\Lambda_0,\Lambda_1;\alpha^\Phi_{01}) \rightarrow L^2(Z^-,\R^{2n})
$$
Here $\Lambda_i$ is considered as Lagrangian subspaces in $\R^{2n}$. Note that this operator is of the form of linearized Cauchy-Riemann operator in coordinates with boundary conditions
$$
Y^\pi(0,t)\in \alpha^\Phi_{01}(t), \; Y^\pi(\tau,0)\in \Lambda_0(\tau), \; Y^\pi(\tau,1)\in \Lambda_1(\tau).
$$

In order to compute the index, we use \cite[Proposition 6.3]{fooo:anchored}. Recall that the Lagrangian loop $\widetilde{\alpha}^\Phi_{[w,\gamma],\alpha_{01}}$ is defined by \eqref{eq:lag-loop-piece}. However in order to use \cite[Proposition 6.3]{fooo:anchored} $\alpha^-(t)$ in \eqref{eq:lag-loop-piece} must satisfy that $(\alpha^-)'(0)$ (instead of $-(\alpha^-)'(1)$) is positively directed. By using the change of variables $(\tau,t) \mapsto (\tau,1-t)$ we can first compute the index of the operator $D\Upsilon_1$ defined on
$W^{1,2}(Z^+,\R^{2n};\Lambda_1,\Lambda_0;\alpha^\Phi_{10})$ and then take the minus sign to the result.
Therefore we have the following proposition.

\begin{prop}
We have
$$ {\rm Index }D\Upsilon_1|_{w^*\xi} = -\mu([w,\gamma];\alpha_{01}). $$
\end{prop}

\subsection{Wrapping them up}

Finally, we have the desired index formula

\begin{thm}\label{thm:index-formula} Let $w$ be a contact instanton satisfying \eqref{eq:contacton-Legendrian-bdy-intro}. Then
we have
\beastar
& &{\rm Index }D\Upsilon(w) = {\rm Index }L_w
= {\rm Index }L_{\rm glued} + \sum_{i=0}^k{\rm Index }L_{w_{(i-1)i^+}} \\
&=& n - \sum_{i=1}^{k}\mu([w_{(i-1)i}^+,\gamma_i];\alpha_{(i-1)i}). \\
\eeastar
\end{thm}

By Proposition \ref{prop:decomp-index}, we can also express it as
$$
{\rm Index }D\Upsilon(w) = n + \mu(\mathcal{E},\underline \gamma, \overline \gamma;B)
$$
where $B = [w]$.

\appendix

\section{Summary of the exterior calculus of vector valued forms}
\label{append:weitzenbock}

In this appendix, we recall the standard Weitzenb\"ock formulas applied to our
current circumstance. A good exposition on the general Weitzenb\"ock formula is
provided in the appendix of \cite{freed-uhlen} and a nice exposition of the covariant
exterior calculus that enters in the proof is given in \cite[Chapter V, Section1]{wells},
for example. Here we just borrow a fragment of this general calculus that are needed
for the purpose of the present paper.

Assume $(P, h)$ is a Riemannian manifold of dimension $n$ with metric $h$,
and $D$ is the Levi--Civita connection.
Let $E\to P$ be any vector bundle with inner product $\langle\cdot, \cdot\rangle$,
and assume $\nabla$ is a connection on $E$ which is compatible with $\langle\cdot, \cdot\rangle$.

For any $E$-valued form $s$, calculating the (Hodge) Laplacian of the energy density
of $s$,  we get
\beastar
-\frac{1}{2}\Delta|s|^2=|\nabla s|^2+\langle \text{\rm Tr} \nabla^2 s, s\rangle,
\eeastar
where for $|\nabla s|$ we mean the induced norm in the vector bundle $T^*P\otimes E$, i.e.,
$|\nabla s|^2=\sum_i|\nabla_{E_i}s|^2$ with $\{E_i\}$ an orthonormal frame of $TP$.
$Tr\nabla^2$ denotes the connection Laplacian, which is defined as
$Tr\nabla^2=\sum_i\nabla^2_{E_i, E_i}s$,
where $\nabla^2_{X, Y}:=\nabla_X\nabla_Y-\nabla_{\nabla_XY}$.

Denote by $\Omega^k(E)$ the space of $E$-valued $k$-forms on $P$. The connection $\nabla$
induces an exterior derivative by
\beastar
d^\nabla&:& \Omega^k(E)\to \Omega^{k+1}(E)\\
d^\nabla(\alpha\otimes \zeta)&=&d\alpha\otimes \zeta+(-1)^k\alpha\wedge \nabla\zeta.
\eeastar
It is not hard to check that for any $1$-forms, equivalently one can write
$$
d^\nabla\beta (v_1, v_2)=(\nabla_{v_1}\beta)(v_2)-(\nabla_{v_2}\beta)(v_1),
$$
where $v_1, v_2\in TP$.
We extend the Hodge star operator to $E$-valued forms by
\beastar
*&:&\Omega^k(E)\to \Omega^{n-k}(E)\\
*\beta&=&*(\alpha\otimes\zeta)=(*\alpha)\otimes\zeta
\eeastar
for $\beta=\alpha\otimes\zeta\in \Omega^k(E)$.

Define the Hodge Laplacian of the connection $\nabla$ by
$$
\Delta^{\nabla}:=d^{\nabla}\delta^{\nabla}+\delta^{\nabla}d^{\nabla},
$$
where $\delta^{\nabla}$ is defined by
$$
\delta^{\nabla}:=(-1)^{nk+n+1}*d^{\nabla}*.
$$
The following lemma is a useful important formulae for the derivation of the Weitzenb\"ock formula.
\begin{lem}\label{lem:d-delta} Assume $\{e_i\}$ is an orthonormal frame of $P$, and $\{\alpha^i\}$ is the dual frame.
Then we have
\beastar
d^{\nabla}&=&\sum_i\alpha^i\wedge \nabla_{e_i}\\
\delta^{\nabla}&=&-\sum_ie_i\rfloor \nabla_{e_i}.
\eeastar
\end{lem}
We refer to \cite[Appendix]{oh-wang:CR-map1} for the proof of Weitzenb\"ock formula
in our context.

In our current context, the triad connection induces the projection map
$$
\nabla^\pi: \Omega_0(\xi) \to \Omega_1(\xi)
$$
by Property Theorem \ref{thm:connection} (4). Then using the CR almost complex structure
$J$ adapted to $\lambda$, we can decompose
$$
\nabla^\pi = \nabla^{\pi(1,0)} + \nabla^{\pi(0,1)}
$$
by the explicit formula
$$
\nabla^{\pi(1,0)} = \frac{\nabla^\pi - J \nabla^\pi_{j(\cdot)}}{2}, \quad
\nabla^{\pi(0,1)} = \frac{\nabla^\pi + J \nabla^\pi_{j(\cdot)}}{2}.
$$
(One may also write $\nabla^{\pi(0,1)}$ as $\overline{\nabla}^\pi$ as done in the main
text.)

Then $d^{\nabla^\pi}$ the associated covariant differential by anti-symmetrizing
it as we define $d^\nabla$ above.
With $\nabla$ above replaced by $\nabla^\pi$, one defines the differential
$$
d^{\nabla^\pi}: \Omega^k(\xi) \to \Omega^{k+1}(\xi).
$$

\begin{rem} \label{rem:triad-connection}
The following two properties are satisfied by the contact triad connection:
\begin{itemize}
\item $J$ is parallel,
\item the $\pi(0,1)$ (or $\pi(1,0)$) projection is an operation parallel with respect to
Hermitian connection $\nabla^\pi$ acting on $\Gamma(w^*\xi)$ for any contact instanton $w$.
\end{itemize}
None of these paralleness  hold for the commonly used Levi-Civita connection.
These properties of the contact triad connection enable us to
organize various terms appearing in the process of covariant differentiation
in an effective tensorial way. Due to the lack of such a canonical decomposition
for the Levi-Civita connection (of the triad metric), one forces to keep track
of derivatives of $J$ and of idempotent $\Pi$ and carry to the end when working with
the Levi-Civita connection: those terms do not admit organization into natural tensorial
expressions, except by themselves!
The latter would prevent the end product of the tensorial calculation from being as
simple and elegant as those given by the contact triad connection, such as in the
calculation leading to the global $W^{2,2}$-estimate or in the formula of the linearlization
$D\Upsilon(w)$. These are a few strong points  of the contact triad connection
over the Levi-Civita connection in the tensorial calculations  given in
\cite{oh-wang:CR-map1,oh-wang:CR-map2,oh:contacton}
and in the present paper.  Similar things occur
between the Chern connection and general Hermitian connection in complex geometry
in a smaller scale. (See \cite{chern:connection}, \cite[Chapter 7]{oh:book1}.)
It is one of the reasons why Wang and the second-named author had
introduced the triad connection in the first place in \cite{oh-wang:connection}
and utilized in \cite{oh-wang:CR-map1,oh-wang:CR-map2}. It
makes the a priori complicated tensorial calculation manageable and produce the end product, e.g.,
the expression of the linearized operator $D\Upsilon(w)$ as given in
\cite[Theorem 10.1]{oh:contacton} has an elegant form which carries
clear geometric meaning. This latter formula  is duplicated in Theorem \ref{thm:linearization}
and utilized in our study of index calculations.
\end{rem}

\bibliographystyle{amsalpha}

\bibliography{biblio2}

\end{document}